\documentclass[preprint,12pt]{elsarticle}

\usepackage{amssymb}
\usepackage{amsmath}
\usepackage{lineno,hyperref}
\usepackage{algorithm}
\usepackage{multirow}
\usepackage{epsfig}
\usepackage{epstopdf}
\usepackage{subfigure}

\newtheorem{theorem}{Theorem}[section]
\newtheorem{corollary}[theorem]{Corollary}

\newtheorem{lemma}[theorem]{Lemma}
\newtheorem{problem}{Problem}
\newtheorem{proposition}[theorem]{Proposition}
\newtheorem{remark}{Remark}[section]
\newtheorem{definition}{Definition}[section]
\newtheorem{assumption}{Assumption}
\newenvironment{proof}[1][Proof]{\textbf{#1.} }

\journal{arXiv}

\begin{document}

\begin{frontmatter}

\title{A CCBM-based generalized GKB iterative regularization algorithm for inverse Cauchy problems\tnoteref{An iterative algorithm for inverse Cauchy problems}}

\author[1,2]{Rongfang Gong}
\affiliation[1]{organization={School of Mathematics, Nanjing University of Aeronautics and Astronautics},
	city={Nanjing},
	postcode={211106},
	country={China}}
\affiliation[2]{organization={Key Laboratory of Mathematical Modelling and High Performance Computing of Air Vehicles (NUAA), MIIT},
	city={Nanjing},
	postcode={211106},
	country={China}}
\ead{grf_math@nuaa.edu.cn}

\author[1]{Min Wang}
\ead{wangmin177@nuaa.edu.cn}
\author[3]{Qin Huang\corref{Corresponding author}}
\affiliation[3]{organization={School of Mathematics and Statistics, Beijing Institute of Technology},
	city={Beijing},
	postcode={100081}, 
	country={China}}
\cortext[Corresponding author]{Corresponding author.}
\ead{huangqin@bit.edu.cn}

\author[3,4]{Ye Zhang}
\affiliation[4]{organization={Shenzhen MSU-BIT University},
	city={Shenzhen},
	postcode={518172}, 
	country={China}}
\ead{ye.zhang@smbu.edu.cn}

\begin{abstract}
	This paper examines inverse Cauchy problems that are governed by a kind of elliptic partial differential equation. The inverse problems involve recovering the missing data on an inaccessible boundary from the measured data on an accessible boundary, which is severely ill-posed. By using the coupled complex boundary method (CCBM), which integrates both Dirichlet and Neumann data into a single Robin boundary condition, we reformulate the underlying problem into an operator equation. Based on this new formulation, we study the solution existence issue of the reduced problem with noisy data. A Golub-Kahan bidiagonalization (GKB) process together with Givens rotation is employed for iteratively solving the proposed operator equation. The regularizing property of the  developed method, called CCBM-GKB, and its convergence rate results are proved under a posteriori stopping rule. Finally, a linear finite element method is used for the numerical realization of CCBM-GKB. Various numerical experiments demonstrate that CCBM-GKB is a kind of accelerated iterative regularization method, as it is much faster than the classic Landweber method.
\end{abstract}


\end{frontmatter}

\section{Introduction}\label{sec:intro}

We are concerned with a Cauchy problem for the Laplace equation where the unknown Cauchy data on a part of the boundary are recovered using the known Cauchy data on the remaining part of the boundary. To be more precise, let $\Omega \subset \mathbb{R}^d$ ($d = {\rm{2}}$ or ${\rm{3}}$ denotes the spatial dimensionality) be an open bounded set with Lipschitz boundary $\Gamma = \partial \Omega $ satisfying $\Gamma = \Gamma_m \cup \Gamma_u$, $\Gamma_m \cap \Gamma_u = \emptyset$, where $\Gamma_m$ and $\Gamma_u$ are known as the accessible boundary and the inaccessible boundary, respectively. Then the Cauchy problem considered in this paper is formulated as follows:
\begin{problem}\label{prob:cauchyproblem}
	Given Cauchy data $(\Phi, T)\in H^{-1/2}(\Gamma_m)\times H^{1/2}(\Gamma_m)$, recover the missing data $(\varphi, t)\in H^{-1/2}(\Gamma_u)\times H^{1/2}(\Gamma_u)$ such that
	\begin{equation}
		\left\{
		\begin{array}{ll}
			- \Delta u = 0 & \textrm{in}\ \Omega,\\
			\partial_n u = \Phi, u = T & \textrm{on}\ \Gamma _m, \\
			\partial_nu = \varphi, u = t & \textrm{on}\ \Gamma _u.
		\end{array} \right.\label{cp}
	\end{equation}
\end{problem}

This kind of identification problem, also known as data completion \cite{AN} has attracted considerable attention from mathematicians, physicists and engineers alike because of its wide-ranging applications in fields such as physics and engineering, specifically in thermostatics \cite{ABN}, plasma physics \cite{JIR}, engineering \cite{Bour}, electrocardiography \cite{CM}, electroencephalogram \cite{Analysis2014} and corrosion non-destructive evaluation \cite{Ingl}, etc.. It is well-known that Problem \ref{prob:cauchyproblem} is ill-posed \cite{BE}. In fact, Ben Belgacem showed in \cite{Belg} that the Cauchy problem is exponentially ill-posed for both smooth and non-smooth domains. We further refer to \cite{ARRV} for an overview of the stability of the Cauchy problem for general elliptic equations.

By introducing an appropriate linear operator $K$ between two infinite dimensional Hilbert spaces $\mathcal{H}_1$ and $\mathcal{H}_2$, Problem \ref{prob:cauchyproblem} can be expressed as:
\begin{equation}
	\label{operator}
	K\phi  = f,
\end{equation}
where $\phi = (\varphi ,t) \in \mathcal{H}_1$, and $f \in \mathcal{H}_2$ is a function that is dependent on the data $\Phi$ and $T$. In the conventional formulation (e.g. according to the language of optimal control), the operator $K$ is usually proposed as the parameter-to-data mapping, i.e. $f=(\Phi ,T)$ in the formulation (\ref{operator}). The ill-posedness of the Cauchy problem is indicated by the compactness of $K$ in infinite dimensional Hilbert spaces. Different mathematical reformulations give different forms of $K$ which, in general, have different null space $\text{Ker}(K)$ and values of singular values $\left\{\sigma_j(K)\right\}_{j \in \mathbb{N}}$. After arranging them in order of decreasing magnitude, it is well known that the ${\sigma _j}(K)$ goes to zero as $j \to + \infty $. Although the convergence rate of $\sigma_j(K)$ to zero behaves similarly even for different forms of $K$, they differ slightly after numerical discretization. Consequently, different expression forms of operator equations (\ref{operator}) lead to systems with different condition numbers. In the same way, different forms of $K$ produce linear systems of different structure, which may affect the stability of the underlying problem and the behavior of numerical methods further used. Hence, the construction of an appropriate operator equation (\ref{operator}) is essential in the numerical solution of Problem \ref{prob:cauchyproblem}.

In this paper, a coupled complex boundary method (CCBM) \cite{Caubet2020} is applied to Problem \ref{prob:cauchyproblem} in order to construct specific $K$ and $f$. The CCBM was originally considered appropriate for inverse source problems in \cite{CGH14}, and was subsequently extended to the inverse Cauchy problem in \cite{CGH16}. Here, we further investigate the properties of CCBM beyond the work of \cite{CGH16}. We skip the systematic analysis and list some merits of CCBM, as follows:
\begin{itemize}
	\item Both the Dirichlet and Neumann conditions are incorporated in a single complex Robin boundary condition, as the imaginary and real parts, respectively, such that the unknown Cauchy data can be reconstructed simultaneously;
	\item The data needed to fit are transferred from $\Gamma_u$ to $\Omega$ which makes the approach to the inverse problem more robust in practice;
	\item No Dirichlet boundary value problem (BVP) needs to be solved, which makes the numerical resolution of the forward problem easier;
	\item In the literature, e.g. \cite{AN,BE,Belg,ARRV}, the Dirichlet data $T$ need to have the regularity $T \in H^{1/2}(\Gamma_m)$ for guaranteeing the necessary regularity of the forward equations. In applications, $T$ is polluted by random noise, hence it is not appropriate to assume $T \in H^{1/2}(\Gamma_m)$. In the CCBM framework, we can allow $T \in L^2(\Gamma_m)$, the natural spaces in practice, which obviates the fractional-order Sobolev functions. At the same time, instead of in $H^{1/2}(\Gamma_u)$, we search for the solution $t$ to the inverse Cauchy problem in $L^2(\Gamma_u)$. Moreover, in CCBM, the noisy data is not used directly as fitting data. Instead, the right-hand side $f$ is also the solution of a BVP which uses the measured Cauchy data as boundary conditions and thus renders a smoothing effect on measurement.
\end{itemize}

Due to the severe ill-posedness of the inverse Cauchy problem and the inevitable random noise in measurement, stabilization strategies (e.g. regularization) are needed for precise reconstruction of the missing functions $(\varphi, t)$. These strategies can be classified into two categories: conventional regularization and regularization by projection (also known as the self-regularization). The conventional regularization methods for the Cauchy problem above mainly feature variational regularization \cite{ABN,ABE}, iterative regularization \cite{KMF,Leit,Marin2020,VLD2021,BCM2022}, truncation regularization \cite{QFL} etc. In this paper, a subspace projection method is applied for the resolution of Problem \ref{prob:cauchyproblem} or operator equation (\ref{operator}) which leads to the finite dimensional optimal problem:
\begin{equation}
	\min_{\phi \in M_k} \|K\phi  - f \|_{\mathcal{H}_2}^2,\label{kylov}
\end{equation}
where $M_k \subset \mathcal{H}_1$ is some $k$-dimensional subspace. In this approach, the dimensionality $k$ of the subspace plays the role of regularization.

Here, we present the main idea of our new regularization algorithm for solving Eq. (\ref{operator}). For fixed $k$, the choice of $M_k$ determines the numerical method. Different choices of $M_k$ entail different numerical algorithms. The realization of $M_k$ is often achieved by the construction of its basis through recursion. To be more precise, let $\{\tilde{u}_j\}_{j \in \mathbb{N}}$ and $\{\tilde{v}_j\}_{j \in \mathbb{N}}$ be the left and right singular functions of operator $K$ associated with the singular values $\{\sigma_j(K)\}_{j \in \mathbb{N}}$, respectively. Let $\phi^\dag$ denote the accurate solution of Eq. (\ref{operator}) in $\mathcal{H}_1$. If we take $M_k = \textrm{span}\{ \tilde{v}_1, \tilde{v}_2, \cdots,\tilde{v}_k\} $, then in the case of noise-free data, the solution of (\ref{kylov}), given by
\begin{equation}
	\phi_k = \sum_{j = 1}^k \frac{(f, \tilde{u}_j)_{\mathcal{H}_2}}{\sigma_j} \tilde{v}_j\label{svd}
\end{equation}
is the best approximation to $\phi^\dag $ among all possible choices of $k$-dimensional subspaces of $\mathcal{H}_1$, and the accuracy in $\phi_k$ improves with increasing values of $k$. The formula (\ref{svd}) is known as the truncated singular value decomposition (TSVD). However, in the presence of noisy data, TSVD may result in a solution that fails to approximate $\phi^\dag$ well, hence $k$ should be chosen carefully. In this paper, a generalized (infinite dimensional) version of Golub-Kahan Bidiagonalization (GKB) process is used to construct the basis $v_1, v_2, \cdots, v_k$ iteratively where the space $\{v_1, v_2, \cdots, v_k\}$ can be taken as a perturbation of $\{ \tilde{v}_1,  \tilde{v}_2, \cdots, \tilde{v}_k\}$. When compared with using the ``best'' $k$-dimensional $ \{\tilde{v}_1,  \tilde{v}_2, \cdots, \tilde{v}_k\}$, the GKB method demonstrates the following advantages:
\begin{itemize}
	\item $v_j$ is produced iteratively and does not, therefore, require any costly singular value decomposition in the resolution process;
	\item The singular functions $\{ \tilde{v}_j\}_{j \in \mathbb{N}}$ depend only on the operator $K$ itself while $v_1, v_2, \cdots $, depend on both $K$ and data $f$, which may restrain ill-posedness to some extent;
	\item Besides playing the role of a regularization parameter, the space dimensionality $k$ also determines the convergence rate of the numerical method. We can expect a more accurate solution via the GKB process, albeit with a smaller value of $k$. This means that we present an accelerated iterative method for the inverse Cauchy problem. The GKB method is also known as the Lanczos Bidiagonalization Method and is equivalent to LSQR for finite dimensional problems. See \cite{PS} for more details.
\end{itemize}

With the GKB method, the original inverse Cauchy problem is reduced to a linear system, associated with a much smaller size coefficient matrix ${G_k}$. Although the decay rate of the singular values of $G_k$ is similar to that of the original forward operator, $K$, the small value of $k$ makes $G_k$ mildly ill-posed. To further suppress the effect of noise in data, a Givens rotation-based QR factorization, see \cite{PS}, is applied to solve the reduced linear system. Compared with Gram-Schmidt or Householder transformations, Givens rotation offers the advantages of space-saving and providing a more stable upper triangular matrix, and is thus more suitable for ill-posed problems.

In sum, with a combination of $K$ with coupled structures, a smoothing $f$, the GKB bidiagonalization as well as Givens rotation, a new fast iterative scheme, called CCBM-GKB, is developed in this paper for Problem \ref{prob:cauchyproblem}, where the space dimensionality $k$, which also represents the iterative step, plays the role of a regularization parameter, and is chosen according to the Morozov discrepancy principle.

The remainder of the paper is organized as follows: In Section \ref{sec:ccbm}, with CCBM, the inverse Cauchy problem is transformed into an operator equation, whose existence and uniqueness are proven. Section \ref{sec:gkb} is devoted to describing a generalized version of the GKB process, followed by some theoretical results, including the regularization property and convergence rate results. In Section \ref{subsec:fem}, a finite element method is applied for the numerical simulation, with the numerical results detailed in Section \ref{subsec:numresult}. Finally, the concluding remarks are outlined in Section \ref{sec:con}.

\section{A reformulation of the Cauchy problem with the CCBM}\label{sec:ccbm}

We introduce first, the notations for function spaces and sets that will feature frequently in this paper. For a set $G$ (e.g., $\Omega$, $\Gamma$), we denote using $W^{m,p}(G)$ the standard Sobolev spaces with norm $\| \cdot \|_{m,p,G}$, $L^p(G) = W^{0,p}(G)$. In particular, $H^m(G)$ represents $W^{m,2}(G)$ with corresponding inner product $(\cdot, \cdot)_{m,G}$ and norm $\|\cdot\|_{m,G}$. Moreover, set $V = H^1(\Omega)$, $Q = L^2(\Omega)$, $Q_\Gamma = L^2(\Gamma)$, $Q_{\Gamma _m} = L^2(\Gamma_m)$, $Q_{\Gamma_u}=L^2(\Gamma_u)$, $Q_m = Q_{\Gamma_m} \times Q_{\Gamma_m}$ and $Q_u = Q_{\Gamma_u} \times Q_{\Gamma_u}$. Throughout this paper, $C$, with or without a subscript, denote a generic constant with a different value for the corresponding setups below. It is dependent only on the geometry of the domain $\Omega$.

Let $i=\sqrt{-1}$ be an imaginary unit. Then, without going into detail, the inverse Cauchy problem (\ref{cp}) can be studied through the following problem \cite{CGH16}.
\begin{problem}\label{prob:2.1}
	With $(\Phi, T)\in Q_m$, find $(\varphi, t)\in Q_u$ such that
	\[
	u_2=0\ \textmd{in}\ \Omega,
	\]
	where $\textbf{u}=u_1+i\,u_2$ solves
	\begin{equation}
		\left\{
		\begin{array}{ll}
			- \Delta \textbf{u} = 0 & \textrm{in}\ \Omega,\\
			\partial_n\textbf{u} + i\,\textbf{u}= \Phi+i\, T & \textrm{on}\ \Gamma _m, \\
			\partial_n\textbf{u} + i\,\textbf{u}= \varphi+i\, t & \textrm{on}\ \Gamma _u.
		\end{array} \right.\label{cbvp}
	\end{equation}
\end{problem}

\begin{remark}\label{remark:2.2}
	Note that $T\in H^{1/2}(\Gamma_m)$ is required for the equivalence between Problems \ref{prob:cauchyproblem} and \ref{prob:2.1}. In the instance that this regularity assumption is not satisfied, the reformulation above, provides a way of extension to Problem \ref{prob:cauchyproblem}. Moreover, in the conventional formulations, the missing Dirichlet data $t$ is often searched for in the fractional-order Sobolev space $H^{1/2}(\Gamma_u)$. However, with the reformulation here, we can look for $t$ in a more natural data space $Q_{\Gamma_u}$.
\end{remark}

For future needs, we introduce the spaces $\textbf{V}=V\otimes i\,V$, the complex version of $V$. Then for any $\textbf{u}=u_1+i\,u_2, \textbf{v}=v_1+i\,v_2\in \textbf{V}$, the inner product of $\textbf{V}$ is $(\textbf{u},\textbf{v})_{1,\Omega}=(u_1,v_1)_{1,\Omega}+(u_2,v_2)_{1,\Omega}$, and the induced norm is $\|\textbf{v}\|_{1,\Omega}=(\textbf{v},\textbf{v})^{1/2}_{1,\Omega}$. The complex version $\textbf{Q}_{\Gamma_u}$ of $Q_{\Gamma_u}$ can be defined in a similar way. For clarity of statement, use $\|\cdot\|_{\textbf{Q}_{\Gamma_u}}$ for the norm of $\textbf{Q}_{\Gamma_u}$.

Define
\begin{eqnarray}
	a(\textbf{u},\textbf{v}) = \int_\Omega \nabla \textbf{u}\cdot \nabla \bar{\textbf{v}}\,dx +	i\int_{\Gamma} \textbf{u}\,\bar{\textbf{v}}ds\quad \forall\, \textbf{u},\textbf{v}\in \textbf{V},\label{auv}\\
	F(\Phi,T,\varphi,t;\textbf{v}) = \int_{\Gamma_m}(\Phi+i\,T)\,\bar{\textbf{v}}\,ds +	\int_{\Gamma_u} (\varphi+i\,t)\, \bar{\textbf{v}}\,ds \quad \forall\,\textbf{v}\in \textbf{V}.\label{fv}
\end{eqnarray}
Then the weak form for the BVP (\ref{cbvp}) is:
\begin{equation}
	\textrm{find}\ \textbf{u}\in \textbf{V},\quad\textrm{ s.t. }a(\textbf{u},\textbf{v})= F(\Phi,T,\varphi,t;\textbf{v}) \quad\forall\,\textbf{v}\in \textbf{V}. \label{weak}
\end{equation}

\begin{proposition}\label{Prop 2.1.}
	\cite[Proposition 2.2]{CGH16} Given $(\Phi,T)\in Q_m$, $(\varphi,t)\in Q_u$, the problem (\ref{weak}) admits a unique solution $\textbf{u}\in \textbf{V}$ which depends continuously on all data. Moreover, there exists a constant $C_1$, depending only on the geometry of the domain, such that
	\begin{equation}
		\|\textbf{u}\|_{1,\Omega}\leq C_1\,(\|\Phi\|_{0,\Gamma_m}+\|T\|_{0,\Gamma_m}+\|\varphi\|_{0,\Gamma_u}+\|t\|_{0,\Gamma_u}).\label{PriEst}
	\end{equation}
\end{proposition}

For $(\Phi ,T)\in Q_m$ and $\phi=\varphi+i\,t\in \textbf{Q}_{\Gamma_u}$, we denote using $\hat{\textbf{u}} =\hat{u}_1+i\,\hat{u}_2\in \textbf{V}$ and $\tilde{\textbf{u}}=\tilde{u}_1+i\,\tilde{u}_2\in \textbf{V}$ the weak solutions of the problems
\begin{equation}
	\left\{
	\begin{array}{ll}
		- \Delta \hat{\textbf{u}} = 0 & \textrm{in}\ \Omega,\\
		\partial_n \hat{\textbf{u}} + i\,\hat{\textbf{u}} = \Phi + iT & \textrm{on}\ \Gamma _m, \\
		\partial_n \hat{\textbf{u}} + i\,\hat{\textbf{u}} = 0 & \textrm{on}\ \Gamma _u,
	\end{array} \right.\label{databvp}
\end{equation}
and
\begin{equation}
	\left\{
	\begin{array}{ll}
		- \Delta \tilde{\textbf{u}} = 0 & \textrm{in}\ \Omega,\\
		\partial_n \tilde{\textbf{u}} + i\,\tilde{\textbf{u}} = 0 & \textrm{on}\ \Gamma _m, \\
		\partial_n \tilde{\textbf{u}} + i\,\tilde{\textbf{u}} = \phi & \textrm{on}\ \Gamma _u.
	\end{array} \right.\label{forwardbvp}
\end{equation}
Then the solution $\textbf{u}$ of the problem (\ref{weak}) satisfies
\[
\textbf{u} = \tilde{\textbf{u}}+ \hat{\textbf{u}}.
\]
View $\tilde{u}_2\in V$ as a function in $Q$ and define a linear operator $K$ from $\textbf{Q}_{\Gamma_u}$ to $Q$,
\begin{equation}
	\phi\to K\phi := \tilde{u}_2. \label{forwardoperator}
\end{equation}
As a result, Problem \ref{prob:2.1} can be transformed further to the following operator equation.
\begin{problem}\label{prob:2.2}
	Given $(\Phi, T)\in Q_m$, set $\hat{\textbf{u}}=\hat{u}_1+i\,\hat{u}_2\in \textbf{V}$ and find $\phi\in \textbf{Q}_{\Gamma_u}$ such that
	\begin{equation}
		K \phi =f := -\hat{u}_2.\label{operatoreqn}
	\end{equation}
\end{problem}

Due to the equivalence of Problem \ref{prob:2.1} and Problem \ref{prob:2.2}, $K$ is compact and Eq. (\ref{operatoreqn}) is ill-posed. In fact, we can explore more properties about $K$. Recall that $K\phi = \tilde{u}_2\in V$ and $\Delta \tilde{u}_2=0$, therefore, the range $R(K)$ of $K$ is a subset of $V_1$, where $V_1$ is a subspace of $V$:
\begin{equation}
	V_1=\{v\in V\vert\ \Delta v=0\},\label{V1}
\end{equation}
that is, the operator $K$ is not surjective. Let $K^*: Q\to \textbf{Q}_{\Gamma_u}$ denote the adjoint operator of $K$. Then for any $v\in Q$, $K^*v=w_2+i\, w_1\vert_{\Gamma_u}$, where $\textbf{w}=w_1+i\,w_2\in \textbf{V}$ is the weak solution of the adjoint problem
\begin{equation}
	\left\{
	\begin{array}{ll}
		- \Delta \textbf{w} = v & \textrm{in}\ \Omega,\\
		\partial_n \textbf{w} + i\,\textbf{w} = 0 & \textrm{on}\ \Gamma.
	\end{array} \right.\label{adjointeqn}
\end{equation}

For the forward operator $K$, we have the following results.
\begin{proposition}\label{prop:K}
	Let $K$ be defined in (\ref{forwardoperator}). Then\par
	(i) $K$ is injective, that is, $\text{Ker}(K) = \{0\}$;\par
	(ii) $K^*$ is not injective;\par
	(iii) $K^*\vert_{V_1}$ is injective;\par
	(iv) $K$ has infinitely many different singular values $\{\sigma_j\}_{j=1}^{+\infty}$ which decay to zero in exponential way.
\end{proposition}
\begin{proof}
	(i) Take any $ \tilde{u}_2, \tilde{v}_2\in R(K)\subset V_1$, and denote using $\phi_1=\varphi_1+i\,t_1, \phi_2=\varphi_2+i\,t_2\in \textbf{Q}_{\Gamma_u}$ the corresponding inverse images, that is,
	\[
	K \phi_1 = \tilde{u}_2, K \phi_2 = \tilde{v}_2.
	\]
	Let $\tilde{\textbf{u}}=\tilde{u}_1+i\,\tilde{u}_2, \tilde{\textbf{v}}=\tilde{v}_1+i\,\tilde{v}_2\in \textbf{V}$ be the weak solutions of the problems (\ref{forwardbvp}), with $\phi$ being replaced by $\phi_1$ and $\phi_2$, respectively. Set $\tilde{\textbf{w}}=\tilde{\textbf{u}}-\tilde{\textbf{v}}=\tilde{w}_1+i\,\tilde{w}_2$. It then holds that
	\[
	\left\{
	\begin{array}{ll}
		- \Delta \tilde{w}_1 = 0 & \textrm{in}\ \Omega,\\
		\partial_n \tilde{w}_1 -\tilde{w}_2 = 0 & \textrm{on}\ \Gamma_m, \\
		\partial_n \tilde{w}_1 -\tilde{w}_2 = \varphi_1-\varphi_2 & \textrm{on}\ \Gamma_u,
	\end{array} \right.
	\]
	and
	\[
	\left\{
	\begin{array}{ll}
		- \Delta \tilde{w}_2 = 0 & \textrm{in}\ \Omega,\\
		\partial_n \tilde{w}_2 +\tilde{w}_1 = 0 & \textrm{on}\ \Gamma_m, \\
		\partial_n \tilde{w}_2 +\tilde{w}_1 = t_1-t_2 & \textrm{on}\ \Gamma_u.
	\end{array} \right.
	\]
	If $\tilde{u}_2=\tilde{v}_2$, then $\tilde{w}_2=0$ in $\Omega$ and thus $\tilde{w}_2=\partial_n \tilde{w}_2=0$ on both $\Gamma_m$ and $\Gamma_u$. Insert these into the BVPs above to give
	\[
	\left\{
	\begin{array}{ll}
		- \Delta \tilde{w}_1 = 0 & \textrm{in}\ \Omega,\\
		\partial_n \tilde{w}_1 = 0 & \textrm{on}\ \Gamma_m, \\
		\tilde{w}_1 = 0 & \textrm{on}\ \Gamma_m.
	\end{array} \right.
	\]
	By using Holmgren's uniqueness theorem \cite{DL,Isakov98}, we have $\tilde{w}_1=0$ in $\Omega$, and thus both $\tilde{w}_1=0$ and $\partial_n \tilde{w}_1=0$ on $\Gamma_u$. Therefore, on $\Gamma_u$,
	\[
	\varphi_1-\varphi_2= \partial_n \tilde{w}_1 -\tilde{w}_2=0 ,\quad t_1-t_2=\partial_n \tilde{w}_2 +\tilde{w}_1=0
	\]
	or $\phi_1=\phi_2$ which means $K$ is injective and thus $\text{Ker}(K) = \{0\}$.\par
	
	(ii) Let $K^*v=(w_2+i\, w_1)\vert_{\Gamma_u}=0$ with $\textbf{w}=w_1+i\,w_2\in \textbf{V}$ being the weak solution of the problem
	\[
	\left\{
	\begin{array}{ll}
		- \Delta \textbf{w} = v & \textrm{in}\ \Omega,\\
		\partial_n \textbf{w} + i\,\textbf{w} = 0 & \textrm{on}\ \Gamma.
	\end{array} \right.
	\]
	Then
	\[
	\left\{
	\begin{array}{ll}
		- \Delta w_1 = v & \textrm{in}\ \Omega,\\
		\partial_n w_1 -w_2 = 0 & \textrm{on}\ \Gamma,
	\end{array} \right.
	\]
	and
	\[
	\left\{
	\begin{array}{ll}
		- \Delta w_2 = 0 & \textrm{in}\ \Omega,\\
		\partial_n w_2 +w_1 = 0 & \textrm{on}\ \Gamma.
	\end{array} \right.
	\]
	Since $w_1=w_2=0$ on $\Gamma_u$, it holds that $\partial_n w_2 =-w_1 = 0$ on $\Gamma_u$, that is, both $w_2=0$ and $\partial_n w_2 =0 $ on $\Gamma_u$. Using Holmgren's uniqueness theorem again, we obtain $w_2=0$ in $\Omega$, which further gives $w_2=\partial_n w_2=0$ on $\Gamma_m$. Therefore, $w_1$ solves the problem:
	\begin{equation}
		\left\{
		\begin{array}{ll}
			- \Delta w_1 = v & \textrm{in}\ \Omega,\\
			w_1  = \partial_n w_1=0 & \textrm{on}\ \Gamma, \\
		\end{array} \right.\label{isp}
	\end{equation}
	which turns out to be an inverse source problem of finding $v\in Q$ from homogeneous Cauchy data. By using \cite[Corollary 2.4]{HCW}, there exist infinitely many pairs of $(w_1, v)\in H^2(\Omega)\times Q$ which solve the problem (\ref{isp}), that is, there has infinitely many $v\in Q$ such that $K^* v=0$, and thus $K^*$ is not injective. \par
	
	(iii) If we further assume $v\in V_1$ satisfying $K^*v=0$, then from (\ref{isp})  and the definition of $V_1$ in (\ref{V1}), there holds
	\[
	\int_\Omega v^2dx = \int_\Omega -\Delta w_1 v dx = \int_\Gamma (\partial_n v w_1  - \partial_n w_1  v)ds=0
	\]
	which gives $v=0$, i.e. $K^*$ is injective when it is restricted over $V_1$.
	
	(iv) Since $\textbf{Q}_{\Gamma_u}$ is a separable space of infinite dimension, and $K$ is injective, the range $R(K)$ is also a separable space of infinite dimension. Therefore, $K$ has infinitely many different singular values $\sigma_j, j=1,2,\cdots,\infty$. This exponentially decaying behavior derives from the severe ill-posedness of the Cauchy problem \cite[Theorem 4.1]{Belg}, and the proof is completed. \qed
\end{proof}

In practice we have only noisy data $(\Phi^\delta, T^\delta)$ at hand. Assume
\[
\|\Phi^\delta - \Phi\|_{0,\Gamma_m} \leq \delta, \quad
\|T^\delta - T\|_{0,\Gamma_m} \leq \delta.
\]
Then the operator equation (\ref{operatoreqn}) is modified to
\begin{equation}
	K \phi = f^\delta,\label{operatoreqnnoise}
\end{equation}
where $f^\delta:= -\hat{u}^\delta_2$ and $\hat{u}^\delta_2$ is the imaginary part of $\hat{\textbf{u}}^\delta$ which is the weak solution of (\ref{databvp}), with $(\Phi, T)$ being replaced by $(\Phi^\delta, T^\delta) $. It is easy to verify that
\begin{equation}\label{right-handsidebound}	
	\|f^\delta-f\|_{0,\Omega}<\|f^\delta-f\|_{1,\Omega}=\|\hat{u}^\delta_2-\hat{u}_2\|_{1,\Omega}\leq C_f\,\delta,
\end{equation}
where $C_f$ is a positive constant that depends only on the geometry of domain.

As a result of Proposition \ref{prop:K}(i), we reach the solution uniqueness as follows.
\begin{corollary}\label{coro:uniqueness1}
	The operator equation (\ref{operatoreqnnoise}) has, at most, one solution.
\end{corollary}

It is well known that in the case the Cauchy data are incompatible, the Cauchy problem lose the existence of the solution. In the following, based on the reformulated problem (\ref{operatoreqnnoise}), we give some existence results. To the end, we introduce the following form: find $\phi\in\textbf{Q}_{\Gamma_u}$ such that
\begin{equation}\label{innerequation}
	(K\phi,K\phi')_{0,\Omega}=(f^\delta,K\phi')_{0,\Omega},\quad \forall\  \phi'\in\textbf{Q}_{\Gamma_u}. 
\end{equation}
Let $K^*K:=\mathcal{A}: \textbf{Q}_{\Gamma_u}\to\textbf{Q}_{\Gamma_u}$. An equivalent statement of (\ref{innerequation}) is to find $\phi\in\textbf{Q}_{\Gamma_u}$ satisfying
\begin{eqnarray}\label{normalequation}
	\mathcal{A}\phi = K^*f^\delta\qquad \textrm{in}\ \textbf{Q}_{\Gamma_u},
\end{eqnarray}
which is the normal equation of the problem (\ref{operatoreqnnoise}). The nature of the problem (\ref{innerequation}) or (\ref{normalequation}) is tightly connected to the properties of the operate $\mathcal{A}$ (i.e. $K^*K$). Therefore, in addition to the properties about $K$ and $K^*$ described in proposition \ref{prop:K}, we also discuss the properties of $K^*K$ below.
\begin{lemma}\label{injective}
	We have that the linear operator $\mathcal{A}: \textbf{Q}_{\Gamma_u}\to\textbf{Q}_{\Gamma_u}$ is one to one.
\end{lemma}

The proof of Lemma \ref{injective} can be found in Appendix A.1. Obviously, $\mathcal{A}$ is a self-adjoint operator. From Lemma \ref{injective}, the kernel $\text{Ker}(\mathcal{A})$ is reduced to the trivial space $\{0\}$. Then, due to $\overline{R(\mathcal{A})}=\text{Ker}(\mathcal{A})^\perp$, the closure of its range coincides with the whole space $\textbf{Q}_{\Gamma_u}$. We summarize them as follows.
\begin{lemma}\label{kerrange}
	There holds
	\begin{eqnarray*}
		\text{Ker}(\mathcal{A}) = \{0\}\quad \textrm{and}\quad \overline{R(\mathcal{A})}=\textbf{Q}_{\Gamma_u}.
	\end{eqnarray*}
\end{lemma}

\begin{lemma}\label{equivalence1}
	Problem (\ref{operatoreqnnoise}) is equivalent to Problem (\ref{normalequation}).
\end{lemma}
\begin{proof}
	We only need to show that $K^*f^\delta\in R(\mathcal{A})$ implies $f^\delta\in R(K)$. To the end, let $K^*f^\delta\in R(\mathcal{A})$. Then there exists a unique solution $\phi$ to the normal Eq. (\ref{normalequation}), that is,
	\begin{eqnarray*}
		K^*K\phi = K^*f^\delta,\quad \text{or equivalently, }\quad K^*(K\phi -f^\delta)=0.
	\end{eqnarray*}
	From Proposition \ref{prop:K}(iii), $K^*$ is injective when restricted to $V_1$. Consequently, by noticing $K\phi, f^\delta\in V_1$, we obtain
	\begin{eqnarray*}
		K\phi -f^\delta=0,
	\end{eqnarray*}
	which gives $f^\delta\in R(K)$.\qed
\end{proof}

If $K^*f^\delta\notin R(\mathcal{A})$, the identity $\mathcal{A}\phi = K^*f^\delta$ fails to have a solution. Therefore, according to \cite{BE}, we can also introduce the definitions of the consistent-pseudo-solutions and the incompatibility measure for the problem (\ref{innerequation}). A sequence $(\phi_n)_n\subset\textbf{Q}_{\Gamma_u}$ is said to be consistent with (\ref{innerequation}) if
\begin{equation}\label{consistentpseudosolution}
	\hspace{-1em}
	\lim_{n\to\infty}\sup_{\phi'\in\textbf{Q}_{\Gamma_u}}\frac{(K\phi_n,K\phi')_{0,\Omega}-(f^\delta,K\phi')_{0,\Omega}}{\|\phi'\|_{\textbf{Q}_{\Gamma_u}}}=\lim_{n\to\infty}\|\mathcal{A}\phi_n-K^*f^\delta\|_{\textbf{Q}_{\Gamma_u}}=0,
\end{equation}
and the incompatibility measure of Eq. (\ref{innerequation}) defined by
\begin{eqnarray*}
	v_S = \inf_{\phi\in\textbf{Q}_{\Gamma_u}}\sup_{\phi'\in\textbf{Q}_{\Gamma_u}}\frac{(K\phi,K\phi')_{0,\Omega}-(f^\delta,K\phi')_{0,\Omega}}{\|\phi'\|_{\textbf{Q}_{\Gamma_u}}}=\inf_{\phi\in\textbf{Q}_{\Gamma_u}}\|\mathcal{A}\phi-K^*f^\delta\|_{\textbf{Q}_{\Gamma_u}}.
\end{eqnarray*} 

Now we get a result on the existence of consistent-pseudo-solution to Eq. (\ref{innerequation}), whether the data are compatible or not.
\begin{proposition}\label{cpsexistence}
	For any $K^*f^\delta\in\textbf{Q}_{\Gamma_u}$, there exists at least a consistent-pseudo-solution $(\phi_n)_n\subset\textbf{Q}_{\Gamma_u}$ of Eq. (\ref{innerequation}). Furthermore, the following alternative holds
	\begin{enumerate}
		\item[(i)] all the consistent-pseudo-solutions $(\phi_n)_n$ blow up in $\textbf{Q}_{\Gamma_u}$, that is
		\begin{eqnarray*}
			\lim_{n\to\infty}\|\phi_n\|_{\textbf{Q}_{\Gamma_u}}=+\infty
		\end{eqnarray*}
		and then $f^\delta\notin R(K)$.
		
		\item[(ii)] at least one consistent-pseudo-solution $(\phi_n)_n$ is bounded in $\textbf{Q}_{\Gamma_u}$, then Eq. (\ref{innerequation}) has a unique solution $\phi\in\textbf{Q}_{\Gamma_u}$ and $f^\delta\in R(K)$.
	\end{enumerate}
\end{proposition}
\begin{proof}
	Let $K^*f^\delta\in\textbf{Q}_{\Gamma_u}$. Then, by Lemma \ref{kerrange}, there exists $(g_n)_{n\in \mathbb{N}}\subset R(\mathcal{A})$ converging to $K^*f^\delta$. We denote by $\phi_n$ the unique solution of Eq. (\ref{innerequation}) where $K^*f^\delta$ is changed to $g_n$. It holds that $\phi_n\in\textbf{Q}_{\Gamma_u}$, for all $n\in\mathbb{N}$, and we have that $\forall\ \phi'\in\textbf{Q}_{\Gamma_u}$,
	\begin{eqnarray}\label{difference}
		&(K\phi_n,K\phi')_{0,\Omega}-(f^\delta,K\phi')_{0,\Omega} = (g_n,\phi')_{0,\Omega}-(f^\delta,K\phi')_{0,\Omega}\\\nonumber
		&=(g_n-K^*f^\delta,\phi')_{0,\Omega} \leq\|g_n-K^*f^\delta\|_{\textbf{Q}_{\Gamma_u}}\|\phi'\|_{\textbf{Q}_{\Gamma_u}}.
	\end{eqnarray}
	Thus, we obtain that
	\begin{eqnarray*}
		\sup_{\phi'\in\textbf{Q}_{\Gamma_u}}\frac{(K\phi_n,K\phi')_{0,\Omega}-(f^\delta,K\phi')_{0,\Omega}}{\|\phi'\|_{\textbf{Q}_{\Gamma_u}}}\leq\|g_n-K^*f^\delta\|_{\textbf{Q}_{\Gamma_u}}.
	\end{eqnarray*}
	Take the limit on both sides of the above inequality, we get that $\phi_n$ satisfies Eq. (\ref{innerequation}). Hence, $(\phi_n)_n\in\textbf{Q}_{\Gamma_u}$ is a solution of (\ref{consistentpseudosolution}) (or a consistent-pseudo-solution of (\ref{innerequation})). 
	
	Let us illustrate point (i) with proof by contradiction. Assume  $f^\delta\in R(K)$. Then $K^*f^\delta\in R(\mathcal{A})$ and Eq. (\ref{innerequation}) admits a solution. Let $\phi$ be the solution of Eq. (\ref{innerequation}). Then we construct a sequence of $\phi_n\equiv\phi$ where is plainly the consistent-pseudo-solution of Eq. (\ref{innerequation}) and bounded. This is a contradiction that all the consistent-pseudo-solutions $(\phi_n)_n$ blow up in $\textbf{Q}_{\Gamma_u}$. 
	
	Here's the point (ii). Assume that $(\phi_n)_n$ is the consistent-pseudo-solution of Eq. (\ref{innerequation}) and bounded in $\textbf{Q}_{\Gamma_u}$. Then, we can extract a subsequence from it, also known as $(\phi_n)_n$, which converges weakly to $\phi\in\textbf{Q}_{\Gamma_u}$. Denote by $g_n=\mathcal{A}\phi_n$. Then from the definition (\ref{consistentpseudosolution}), there holds:
	\[
	\lim_{n\to\infty}\|g_n-K^*f^\delta\|_{\textbf{Q}_{\Gamma_u}}=0,
	\]
	which gives
	\[
	(g_n-K^*f^\delta,\phi')_{0,\Omega}=0,\quad \forall\ \phi'\in\textbf{Q}_{\Gamma_u}.
	\]
	Passing to the limit in (\ref{difference}) and thanks to the continuity of the operator $K$ and $(\cdot,K\phi')$, $\phi$ is the unique solution of Eq. (\ref{innerequation}) and thus $K^*f^\delta\in R(\mathcal{A})$. According to Lemma \ref{equivalence1}, we have $f^\delta\in R(K)$.  \qed
\end{proof}

Proposition \ref{cpsexistence} indicates that although the consistent-pseudo-solution of Eq. (\ref{innerequation}) always exists for any measurement data $f^\delta$, the solution to it may not exist when $f^\delta$ is not in range. Note that practically,  the approximate solution is usually searched in a finite dimensional space. Therefore, in the following, we will discuss the solvability of Eq. (\ref{innerequation}) in the finite dimensional space. By the Ritz-Galerkin procedure, Eq. (\ref{innerequation}) over the finite-dimensional subspace $M_m\subset \textbf{Q}_{\Gamma_u}$ with the dimensionality $m$ can be formulated: find $\phi_m\in M_m$ such that
\begin{eqnarray}\label{finitedimproblem}
	(K\phi_m,K\phi')_{0,\Omega}=(f^\delta,K\phi')_{0,\Omega},\quad \forall\ \phi'\in M_m.
\end{eqnarray}
Obviously, we have that
\begin{eqnarray*}
	(K\phi',K\phi')_{0,\Omega}\geq\alpha_m\|\phi'\|^2_{\textbf{Q}_{\Gamma_u}},\quad \forall\ \phi'\in M_m.
\end{eqnarray*}
($\alpha_m$ decays to zero for growing the dimensionality of $M$). As a result, by Lax-Milgram lemma, problem (\ref{finitedimproblem}) admits one solution $\phi_m\in M_m$ for fixed $m$. Moreover, about the approximate solution sequence $\{\phi_m\}_{m=1}^\infty$, it is easy to prove that the following properties hold.
\begin{proposition}\label{converge1}
	Denote by $\phi_m\subset M_m$ the unique solution to the problem  (\ref{finitedimproblem}), we have:
	\begin{itemize}
		\item[(i)]  $\{\phi_m\}_{m=1}^\infty$ is a consistent-pseudo-solution  of Eq. (\ref{innerequation});
		\item[(ii)] if $f^\delta\notin R(K)$,  $\{\phi_m\}_{m=1}^\infty$ blows up in $\textbf{Q}_{\Gamma_u}$;
		\item[(iii)]  if $f^\delta\in R(K)$ and $M_m\to \textbf{Q}_{\Gamma_u}$ as $m\to \infty$, 
		\[
		\phi_m\to \phi^\delta\quad in\ \textbf{Q}_{\Gamma_u},
		\]
		where $\phi^\delta$ is the unique solution to Eq. (\ref{innerequation}).
	\end{itemize}
\end{proposition}

We end this section with an important property of the residual functional, which is defined as
\begin{eqnarray*}
	J^\delta(\phi) = \frac{1}{2}\|u_2(\phi;T^\delta,\Phi^\delta)\|^2_{0,\Omega},\quad \forall\ \phi\in\textbf{Q}_{\Gamma_u}.
\end{eqnarray*}

\begin{lemma}\label{objective1}
	Let $(T^\delta,\Phi^\delta)$ be arbitrarily given in $Q_{m}$, then we have that
	\begin{eqnarray}\label{function}
		\inf_{\phi\in\textbf{Q}_{\Gamma_u}}J^\delta(\phi)=0.
	\end{eqnarray}
\end{lemma}
\begin{proof}
	(i) When the data $(T^\delta,\Phi^\delta)$ are compatible, Problem \ref{prob:2.1} exists a solution and thus $u_2$ equals zero. Then, formula (\ref{function}) is obviously true.
	
	(ii) Considering the data $(T^\delta,\Phi^\delta)$ are incompatible, i.e. $K^*f^\delta\notin R(\mathcal{A})$. Let us construct a minimizing sequence for $J^\delta$. According to \cite[Proposition 2.3]{BE}, there exists $(T^\delta_n)_n\subset Q_{\Gamma_m}$ converging towards $T^\delta$ such that $\{T^\delta_n,\Phi^\delta\}$ are compatible. Let $\phi_n$ denote the solution of the associated Eq. (\ref{innerequation}) (where $T^\delta$ is changed to $T^\delta_n$), thus we have $u_2(\phi_n;T^\delta_n,\Phi^\delta)=0$. Then the sequence $(\phi_n)_n$ is minimizing $J^\delta$. Indeed, it holds that
	\begin{eqnarray*}
		&J^\delta(\phi_n) =\frac{1}{2}\|u_2(\phi_n;T^\delta,\Phi^\delta)\|^2_{0,\Omega}= \frac{1}{2}\|u_2(\phi_n;T^\delta,\Phi^\delta)-u_2(\phi_n;T^\delta_n,\Phi^\delta)\|^2_{0,\Omega}\\
		&= \frac{1}{2}\|u_2(0;T^\delta-T^\delta_n,0)\|^2_{0,\Omega}\leq\frac{1}{2}\|\textbf{u}(0;T^\delta-T^\delta_n,0)\|^2_{0,\Omega}\leq C\|T^\delta-T^\delta_n\|^2_{0,\Gamma_m}.
	\end{eqnarray*}
	The last inequality is obtained from Proposition \ref{Prop 2.1.}. This proves that $(J^\delta(\phi_n))_n$ goes to zero and since $J^\delta$ is non-negative, it comes out that the conclusion is true.\qed
\end{proof}

\section{The generalized GKB iterative method for the operator equation}\label{sec:gkb}

Let the exact data $(\Phi, T)$ be compatible and let $(\varphi^\dag, t^\dag)\in H^{-1/2}(\Gamma_u) \times H^{1/2}(\Gamma_u)$ be the unique solution of the original Cauchy problem with the exact data.
\begin{assumption}\label{assump:2}
	Assume $\phi^\dag=\varphi^\dag+i\,t^\dag\in \textbf{Q}_{\Gamma_u}$.
\end{assumption}
Immediately, from Assumption \ref{assump:2}, the true solution $\phi^\dag$ is the unique solution to the operator equation (\ref{operatoreqnnoise}) of noise-free data. With the given data $(\Phi^\delta, T^\delta)$, in this section, we are devoted to proposing an iterative algorithm for computing approximations to $\phi^\dag$.

\subsection{The generalized GKB process}

The generalized GKB process\footnote{The conventional GKB process is proposed only in finite dimensional vector spaces.} for problem (\ref{operatoreqnnoise}) reads as follows:
\begin{equation}
	\left\{
	\begin{array}{l}
		\gamma_1 = \|f^\delta\|_{0,\Omega},\  \gamma_1 v_1 =  f^\delta,\\[0.5em]
		\beta_1 = \|K^*v_1\|_{\textbf{Q}_{\Gamma_u}},   \beta_1 p_1 =  K^*v_1,\\[0.5em]
		\gamma_{j+1} =\|K p_j - \beta_j v_j\|_{0,\Omega},\ \gamma_{j+1} v_{j+1} = K p_j - \beta_j v_j, \\[0.5em]
		\beta_{j+1} =\|K^* v_{j+1} - \gamma_{j+1} p_j \|_{\textbf{Q}_{\Gamma_u}}, \  \beta_{j+1} p_{j+1} = K^* v_{j+1} - \gamma_{j+1} p_j,\\[0.5em]
		j = 1,2,3,\cdots.
	\end{array} \right.\label{gkb}
\end{equation}

\begin{remark}
	As shown in the Appendices, in conventional frameworks, fractional-order Sobolev spaces need to be introduced for necessary solution regularity. Specifically, in Appendix A.2, the forward operator $K$ is defined from $H^{1/2}({\Gamma_u})$ to $Q_{\Gamma_m}$; in Appendix A.3, the forward operator $K$ is defined from	$H^{-1/2}({\Gamma_u})\times H^{1/2}({\Gamma_u})$ to $Q$. As a result, when implementing the GKB process	(\ref{gkb}), we have to compute the norms $\|\cdot\|_{1/2,\Gamma_u}$ or $\|\cdot\|_{-1/2,\Gamma_u}$ which makes the numerical computation more challenging.
\end{remark}

Throughout this paper, the following assumption is made.
\begin{assumption}\label{assumpProcedure}
	For infinite dimensional problem (\ref{operatoreqnnoise}), the GKB process for the reduced operator equations does not stop in finite steps unless some termination criteria is introduced.
\end{assumption}

From the definition and the method of induction, it is not difficult to show the following result:
\begin{lemma}
	Let Assumption \ref{assumpProcedure} hold, and $\{p_j\}_{j=1}^{\infty}, \{v_j\}_{j=1}^{\infty}$ be obtained through (\ref{gkb}). Then they are normalized orthogonal in spaces $\textbf{Q}_{\Gamma_u}$ and $Q$, respectively.
\end{lemma}

\begin{definition}
	Let $\lambda$ be a vector in $\mathbb{R}^k$, $G$ be a matrix in $\mathbb{R}^{k \times k}$, and $A$ be a linear operator with domain contained in some Hilbert space $H$. Denote $X_k = [x_1,\cdots,x_k]$ with $x_j\in H, j=1,\cdots, k$, and define
	\begin{eqnarray*}
		X_k \star \lambda := \sum\limits_{j=1}^k \lambda_j x_j, \quad A X_k := [A x_1,A x_2,\cdots,A x_k], \\
		X_k \star G := [X_k \star G(:,1),X_k \star G(:,2),\cdots, X_k \star G(:,k)],
	\end{eqnarray*}
	where $G(:,j)$ represents the j-th column of the matrix $G$.
\end{definition}

From the definition above, it is easy to verify the following relation:
\begin{equation}
	(X_k\star  G)\star d = X_k\star (Gd).\label{starrelation}
\end{equation}

With the introduction of the operations above, the generalized GKB process (\ref{gkb}) can be written in a form of matrix-vector type:
\begin{equation}
	\left\{
	\begin{array}{l}
		V_{k+1} \star (\gamma_1 e_1) = f^\delta,\\[0.5em]
		K P_k = V_{k+1} \star G_k,\\[0.5em]
		K^* V_{k+1} = P_k \star G_k^T + \beta_{k+1} p_{k+1} e_{k+1}^T,
	\end{array} \right.\label{gkbmatix}
\end{equation}
where $e_i$ is the $i$th normal unit vector, $P_k = [p_1,\cdots,p_k], V_k = [v_1,\cdots,v_k]$, $G_k^T$ is the transpose of $G_k$ with
\begin{eqnarray}
	G_k =
	\left(
	\begin{array}{cccc}
		\beta_1 & & & \\
		\gamma_2 & \beta_2 & & \\	
		& \ddots & \ddots & \\
		& & \gamma_k & \beta_k \\
		& & & \gamma_{k+1}
	\end{array} \right)_{(k+1)\times k}.\label{matrixG}
\end{eqnarray}
Moreover, we have the following isometric result:
\begin{proposition}\label{prop:isometric}
	Let $\{v_j\}_{j=1}^{\infty}$ be produced through (\ref{gkb}). We denote by $V_k = [v_1,v_2,\cdots,v_k]$ and $\|\cdot\|_2$ the Euclid norm of $\mathbb{R}^k$. Then for any $\lambda\in \mathbb{R}^k$, it holds that $\|V_k \star \lambda\|_{0,\Omega} = \|\lambda\|_2$.
\end{proposition}
\begin{proof}
	Due to the normalization and orthogonality of $\{v_j\}_{j=1}^{\infty}$, the proposition follows from
	\begin{align*}
		\|V_k \star \lambda\|^2_{0,\Omega}&=\int_\Omega\left(\sum_{j=1}^k \lambda_j v_j\right)^2dx  =\int_\Omega \sum_{j=1}^k \lambda_j^2 v_j^2dx\\
		& =\sum_{j=1}^k \lambda_j^2\int_\Omega  v_j^2dx=\sum_{j=1}^k \lambda_j^2=\|\lambda\|_2^2.
	\end{align*}\qed
\end{proof}

Let $\{p_j\}_{j=1}^{\infty}$ and $\{v_j\}_{j=1}^{\infty}$ be produced through (\ref{gkb}). Then, by using arguments similar to those of \cite[Proposition 3.8]{KJ}, we can prove that for each $k\in \mathbb{N}$, $\{p_j\}_{j=1}^{k}$ and $\{v_j\}_{j=1}^{k}$ are the orthonormal bases of the subspaces $\mathbb{K}_k(K^*K;K^*f^\delta)$ and $\mathbb{K}_k(KK^*;f^\delta)$, respectively, where the Krylov subspace $\mathbb{K}_k(L;g)$ for a linear operator $L$ in a space $H$ and an element $g\in H$ is defined as $\mathbb{K}_k(L;g)=\textrm{span}\{g, Lg, L^2 g, \cdots, L^{k-1}g\}$.

Now we are in a position to present the approximation solutions to the exact solution $\phi^\dag$. With a fixed $k\in \mathbb{N}$, we solve the optimal problem in finite dimension subspaces $\mathbb{K}_k(K^*K;p_1)$ for the approximate solution:
\begin{eqnarray}\label{minimizationproblem}
	\min\limits_{\phi \in \mathbb{K}_k(K^*K;p_1)}\|K\phi-f^\delta\|_{0,\Omega}^2. \label{proj}
\end{eqnarray}
According to the solvability of the problem (\ref{finitedimproblem}), we conclude that the problem (\ref{minimizationproblem}) exists a unique solution. For any $\phi_k \in \mathbb{K}_k(K^*K;p_1)$, we have the expansion
\begin{equation}
	\phi_k = P_k \star \lambda,\label{expansion}
\end{equation}
for some $\lambda \in {\mathbb{R}^k}$, where $P_k = [p_1,\cdots,p_k]$. Due to (\ref{gkbmatix}),
\begin{eqnarray*}
	K\phi_k-f^\delta  = (V_{k+1} \star G_k) \star \lambda-V_{k+1} \star (\gamma_1 e_1)  = V_{k+1} \star (G_k \star \lambda-\gamma_1 e_1).
\end{eqnarray*}
Then, by using Proposition \ref{prop:isometric}, we have 
$$\|K\phi-f^\delta \|_{0,\Omega}^2 = \|G_k\lambda-\gamma_1 e_1 \|_2^2.$$ 
As a result, the problem (\ref{proj}) is reduced to the following problem:
\begin{equation}
	\min\limits_{\lambda \in {\mathbb{R}^k}} \|G_k \lambda - \gamma_1 e_1\|_2^2,
	\label{opteqn}
\end{equation}

Since $G_k$ is column full rank, the minimizer $\lambda^*$ of the optimal problem (\ref{opteqn}) exists uniquely. Once $\lambda^*$ is obtained, through the expression (\ref{expansion}), the approximate solution $\phi_k^\delta=P_k\star\lambda^*$ is arrived at.

It is claimed in \cite{GLO,OS} that for discrete $K$, the matrix $G_k$ usually contains very good approximations to the large singular values and rough approximations to the small ones of $K$, and $G_k$ is typically ill-conditioned. In \cite{OS}, the TSVD is used for the regularization of the problem (\ref{opteqn}). For the infinite compact operator $K$ here, by combining Proposition \ref{prop:K}(iv) and \cite[Proposition 1]{CN19}, we obtain the following result for the elements of matrix $G_k$:
\begin{lemma}\label{prop:Gk}
	Both $\{\beta_j\}_{j=1}^{+\infty}$ and $\{\gamma_j\}_{j=1}^{+\infty}$ converge to zero.
\end{lemma}

Proof of the above lemma can be found in Appendix A.4. As a corollary of Lemma \ref{prop:Gk} above, it is indicated that the problem (\ref{opteqn}) is ill-posed when $k$ is large.
In fact, for any $k\in \mathbb{N}$, let $\sigma_{k,k}<\sigma_{k,k-1}<\cdots<\sigma_{k,2}<\sigma_{k,1}$ be all the singular values of $G_k$. Then we have
\[
\sigma_{k,1}\geq \sqrt{\beta_1^2+ \gamma_2^2},\ \sigma_{k,k}\leq \sqrt{\beta_k^2+ \gamma_{k+1}^2},
\]
which means the spectral condition number of $G_k$
\[
\kappa_2(G_k)=\frac{\sigma_{k,1}}{\sigma_{k,k}}\geq \sqrt{\frac{\beta_1^2+ \gamma_2^2}{\beta_k^2+ \gamma_{k+1}^2}}\to \infty
\]
increases in an exponential way as $k\to \infty$. 

In \cite{KJ}, the authors apply the GKB process to the problem
\[
(K+\mu I) \phi =f^\delta
\]
for the regularization with $\mu>0$ being the regularization parameter. In this paper, however, the GKB process is directly used for the problem (\ref{minimizationproblem}). Moreover, no further regularization strategy is adopted for the resolution of the problem (\ref{opteqn}). Instead, a QR decomposition is taken for computing the minimizer $\lambda^*$. The reasons are threefold. Firstly, with the CCBM, the operator $K$ of specific coupled structure and smoothing $f^\delta$ are produced which improve the numerical stability. Secondly, compared with solving the normal equation, applying QR decomposition for the least square problem results in solutions which are more numerically stable and of greater accuracy, and this is especially true when the underlying problem is ill-posed. Thirdly, as shown by the numerical results in Section \ref{subsec:numresult}, the iterative step $k$ is usually small, which makes $G_k$ mildly ill-conditioned.

In broad terms, there are three orthogonal transformations for the purpose of QR decomposition: Gram-Schmidt transformation, Householder transformation and Givens transformation. Due to its advantages in space-saving and computational efficiency, Givens transformation is applied for the resolution of the problem (\ref{opteqn}). Specifically, for $k=1$, let $\bar{\tau}_1=\beta_1, \bar{\mu}_1= \gamma_1, \tau_1= \sqrt{\bar{\tau}_1^2 + \gamma_{2}^2}$, $c_1 = \frac{\bar{\tau}_1}{\tau_1}$, $s_1 = \frac{\gamma_{2}}{\tau_1}$, and
\[
Q_1=\left(
\begin{array}{ll}
	c_1 & s_1 \\
	-s_1 & c_1
\end{array} \right).
\]
Then
\[
Q_1G_1=
\left(
\begin{array}{l}
	\tau_1 \\
	0
\end{array} \right),\quad
Q_1 \gamma_1 e_1=
\left(
\begin{array}{l}
	\mu_1  \\
	\bar{\mu}_2
\end{array} \right)
:=
\left(
\begin{array}{l}
	c_1 \bar{\mu}_1 \\
	-s_1 \bar{\mu}_1
\end{array} \right)
\]
and (\ref{opteqn}) is reduced to
\[
\min\limits_{\lambda \in {\mathbb{R}}}
\|\left(
\begin{array}{l}
	\tau_1 \\
	0
\end{array} \right)\lambda - \left(
\begin{array}{l}
	\mu_1  \\
	\bar{\mu}_2
\end{array} \right)\|_2^2 ,
\]
which gives $\lambda_1^*=\mu_1/\tau_1$. As a result, the approximate solution $\phi_1=\lambda_1^* p_1$, and the residual $\|K\phi_1-f^\delta \|_{0,\Omega} = \|G_1\lambda_1^*-\gamma_1 e_1 \|_2=\vert\bar{\mu}_2\vert$.

For $k=2$, let further $\bar{\tau}_2=c_1 \beta_2 , \eta_2=s_1 \beta_2, \tau_2= \sqrt{\bar{\tau}_2^2 + \gamma_{3}^2}$, $c_2 = \frac{\bar{\tau}_2}{\tau_2}$, $s_2 = \frac{\gamma_{3}}{\tau_2}$, and
\[
\tilde{Q}_1=\left(
\begin{array}{lll}
	c_1 & s_1 & 0\\
	-s_1 & c_1 & 0\\
	0   & 0   & 1
\end{array} \right), \tilde{Q}_2=\left(
\begin{array}{lll}
	1 &  0   & 0  \\
	0 &  c_2 & s_2\\
	0 & -s_2 & c_2
\end{array} \right), Q_2= \tilde{Q}_2\tilde{Q}_1.
\]
Then
\[
Q_2G_2=
\left(
\begin{array}{ll}
	\tau_1 & \eta_2\\
	0     & \tau_2\\
	0     & 0
\end{array} \right),\quad
Q_2 \gamma_1 e_1=
\left(
\begin{array}{l}
	\mu_1  \\
	\mu_2\\
	\bar{\mu}_3
\end{array} \right)
:=
\left(
\begin{array}{l}
	\mu_1 \\
	c_2 \bar{\mu}_2\\
	-s_2 \bar{\mu}_2
\end{array} \right)
\]
and (\ref{opteqn}) is reduced to
\[
\min\limits_{\lambda \in {\mathbb{R}^2}}
\|\left(
\begin{array}{ll}
	\tau_1 & \eta_2\\
	0     & \tau_2\\
	0     &  0
\end{array} \right)\lambda - \left(
\begin{array}{l}
	\mu_1  \\
	\mu_2\\
	\bar{\mu}_3
\end{array} \right)\|_2^2 ,
\]
which gives $\lambda_2^*=(\lambda_{2,1}^*,\lambda_{2,2}^*)^T$ with $\lambda_{2,1}^*=(\mu_1-\mu_2\eta_2/\tau_2)/\tau_1$, $\lambda_{2,2}^*=\mu_2/\tau_2$. As a result, the approximate solution is
\[
\phi_2=\lambda_{2,1}^* p_1+\lambda_{2,2}^* p_2=\lambda_1^* p_1 + \frac{\mu_2}{\tau_2} q_2
\]
with $q_2=p_2-\eta_2/\tau_1 q_1$ and $q_1=p_1$. Correspondingly, the residual $\|K\phi_2-f^\delta \|_{0,\Omega} = \|G_2\lambda_2^*-\gamma_1 e_1 \|_2=\vert\bar{\mu}_3\vert$.

Continuing the process for $k=3, 4, \cdots$, we derive an iterative scheme for producing a sequence of approximate solutions, as shown in \textbf{Algorithm 1}.

\begin{algorithm}[H] \label{algm1}
	\caption{The generalized GKB method for the operator equation (\ref{operatoreqnnoise}).}
	\noindent \textbf{Initialize} \\
	1. Set $\phi_0^\delta=0$;\\
	2. Compute $\gamma_1 = \|f^\delta\|_{0,\Omega}$, $ v_1 =  f^\delta/\gamma_1$,
	$\beta_1 = \|K^*v_1\|_{\textbf{Q}_{\Gamma_u}}$, $p_1 =  K^*v_1/\beta_1 $; \\
	3. Set $q_1=p_1, \bar{\mu}_1=\gamma_1, \bar{\tau}_1=\beta_1$;\\
	\textbf{The GKB process } \\
	For $j = 1, 2, \cdots, $ until stopping: \\
	4. $\gamma_{j+1} =\|K p_j - \beta_j v_j\|_{0,\Omega}$,  $v_{j+1} = (K p_j - \beta_j v_j)/\gamma_{j+1}$;\\
	5. $\beta_{j+1} =\|K^* v_{j+1} - \gamma_{j+1} p_j \|_{\textbf{Q}_{\Gamma_u}}$, $ p_{j+1} = (K^* v_{j+1} - \gamma_{j+1} p_j)/\beta_{j+1}$;\\
	\textbf{QR decomposition together with Givens transformation} \\
	6. $\tau_j= \sqrt{\bar \tau_j^2 + \gamma_{j+1}^2}, c_j = \frac{\bar{\tau}_j}{\tau_j}, s_j = \frac{\gamma_{j+1}}{\tau_j}$;\\
	7. $\eta_{j+1} = s_i\beta_{j+1}, \bar{\tau}_{j+1} = c_j\beta_{j+1}, \mu_j = c_j\bar{\mu}_j, \bar{\mu}_{j+1} = -s_j\bar{\mu}_i$;\\
	8. $\phi_j^\delta = \phi_{j-1}^\delta + \frac{\mu_j}{\tau_j}q_j, q_{j+1} = p_{j+1} - \frac{\eta_{j+1}}{\tau_j}q_j $;\\
	\textbf{Residual computation} \\
	9. $\vert\bar{\mu}_{j+1}\vert$.\\
	\textbf{Output}\\
	10. $\phi_j^\delta$.
\end{algorithm}

The GKB process, together with Givens transformation, is known as LSQR algorithm, and has been applied frequently to discrete linear systems \cite{PS,OS}. A continuous version of the GKB method, together with a variational regularization strategy for the linear Fredholm integral equations, is considered in \cite{KJ}. The method for general compact operator equations with noise-free data is also studied in \cite{CN19,NP}. This paper differs from all these other works in that noisy data is given consideration while no additional regularization parameter is introduced. Instead, only the iterative step $k$ plays the role of the regularization parameter.

Next we study the convergence properties of the generalized GKB method. For this purpose let $\mathcal{K}(K^*K;K^*f^\delta)$ be the closed linear span of the vectors $\{(K^*K)^k K^*f^\delta\}_{k=0}^\infty$, that is,
\[
\mathcal{K}(K^*K;K^*f^\delta)=\lim_{k\to \infty}\mathbb{K}_k(K^*K;K^*f^\delta).
\]
For a fixed $\delta\geq 0$, denote by $\phi^\delta_k\in \mathbb{K}_k(K^*K;p_1)=\mathbb{K}_k(K^*K;K^*f^\delta)$ the unique solution of the problem (\ref{proj}). Then we have the following results.

\begin{proposition}\label{prop:conv}
	For the approximate solution sequence $\{\phi^\delta_k\}_{k=1}^\infty$, there holds:
	\begin{itemize}
		\item[(i)]  $\{\phi^\delta_k\}_{k=1}^\infty$ is a consistent-pseudo-solution  of Eq. (\ref{innerequation});
		\item[(ii)] $\{\phi^\delta_k\}_{k=1}^\infty$ is a minimizing sequence of the objective function $J^\delta(\cdot)$ in $\textbf{Q}_{\Gamma_u}$;
		\item[(iii)] if $f^\delta\notin R(K)$,  $\{\phi_k^\delta\}_{k=1}^\infty$ blows up in $\textbf{Q}_{\Gamma_u}$.;
		\item[(iv)]  if $f^\delta\in R(K)$, $
		\phi_k^\delta\to \phi^\delta\  in\ \textbf{Q}_{\Gamma_u}$,	where $\phi^\delta$ is the unique solution to Eq. (\ref{operatoreqnnoise}).
	\end{itemize}
\end{proposition}
\begin{proof}
	Recall that $\phi^\delta_k$ solves (\ref{finitedimproblem}) with $M_m$ being replaced by $\mathbb{K}_k(K^*K;K^*f^\delta)$. Therefore, conclusions (i) and (iii) are derived directly from Proposition \ref{converge1} (i) and (ii). Conclusion (iii) can be verified by combing the definition of consistent-pseudo-solution and Lemma \ref{objective1}.
	
	For (iv), since $f^\delta\in R(K)$, and $\phi^\delta\in \textbf{Q}_{\Gamma_u}$ solves (\ref{operatoreqnnoise}) or (\ref{normalequation}). Then applying \cite[Theorem 1]{CN19}, we have $\phi^\delta\in \mathcal{K}(K^*K;K^*f^\delta)$. As a result, by using Proposition \ref{converge1} (iii), and recognizing the relation of $\mathcal{K}(K^*K;K^*f^\delta)$	and $\mathbb{K}_k(K^*K;K^*f^\delta)$, we conclude that $\phi^\delta_k\to \phi^\delta$ in $\textbf{Q}_{\Gamma_u}$ as $k\to \infty$.\qed
\end{proof}

Applying Proposition \ref{prop:conv} (iv), in the case of noise-free data $f$ ($f^\delta$ with $\delta=0$), we have the convergence of $\phi_k$ ($\phi^\delta_k$ with $\delta=0$) to the exact solution $\phi^\dag$ ($\phi^\delta$ with $\delta=0$). However, according to Proposition \ref{prop:conv}  (iii), when $\delta>0$ and the data is incompatible, $\{\phi_k^\delta\}_{k\geq 1}$ blows up. In fact, in this case, the semi-convergence of $\phi^\delta$ to $\phi^\dag$ occurs (\cite[Theorem 3.1]{Jia}), and the iterative step $k$ must be chosen properly.

\subsection{Discrepancy principle and error estimation}

We formulate the regularization property of \textbf{Algorithm 1} as follows.
\begin{theorem}
	\label{thmRegu}
	Let Assumptions \ref{assump:2}--\ref{assumpProcedure} hold. Let $\phi_{k(\delta)}^\delta$ be constructed by the generalized GKB method with a stopping rule, such that $k(\delta)\to \infty$ as $\delta\to0$. Then, \textbf{Algorithm 1} yields a regularization algorithm, i.e. $\phi_{k(\delta)}^\delta \to \phi^\dag$ as $\delta\to0$.
\end{theorem}

The Morozov discrepancy principle is applied for this purpose in this paper.

\noindent \textbf{Discrepancy principle:} Given $\tau>0$, implement \textbf{Algorithm 1} for $j=1,2,\cdots,k(\delta)$, with $k=k(\delta)$ being the smallest index, such that
\begin{equation}
	\label{dp}
	\|K\phi_k^\delta-f^\delta\|_{0,\Omega}=\vert\bar{\mu}_{k+1}\vert\leq \tau \delta.
\end{equation}

From  Proposition \ref{prop:conv} (ii),  $\lim\limits_{k\to \infty}\vert\bar{\mu}_{k}\vert=0$. In fact, we have a more precise result for the residual.
\begin{proposition}\label{prop:residual}
	For any $\delta\geq 0$, we denote by $\{\phi_j^\delta\}_{j\geq 1}$ the approximate solutions produced by \textbf{Algorithm 1}. Then the corresponding residuals $\{\vert\bar{\mu}_{j+1}\vert\}_{j\geq 1}$ are strictly monotonically decreasing and converge to 0 as $j\to \infty$.
\end{proposition}
\begin{proof}
	We only need to prove the monotonicity. Recall that $\beta_j, \gamma_j>0$ for all $j\in \mathbb{N}$. Then, by using mathematical induction, it is easy to get $0<\bar{\tau}_j\leq \beta_j$ for all $j\in \mathbb{N}$. Therefore, $0< c_j <1, 0<s_j <1$. From the iteration, it holds true that for $j\geq 1$,
	\[
	\vert\bar{\mu}_{j+1}\vert = s_j\vert\bar{\mu}_j\vert<\vert\bar{\mu}_j\vert,
	\]
	which gives the strict monotonicity. Moreover, with the mathematical induction again, we have
	\begin{equation}
		\vert\bar{\mu}_j\vert>0, \ j\geq 1.\label{nonnegative}
	\end{equation}\qed
\end{proof}

\begin{corollary}\label{Corollary:3.8}
	With a fixed positive parameter $\tau$, then the iteration index $k(\delta)$ determined by the discrepancy principle satisfies
	\[
	k(\delta)\to \infty\ \textrm{as} \ \delta\to 0.
	\]
	Consequently, \textbf{Algorithm 1} with the discrepancy principle stopping rule (\ref{dp}) yields a regularization algorithm.
\end{corollary}

\begin{remark}\label{remark:3.9}
	For the conventional regularization methods, we need to assume that $k(\delta)\to \infty \textrm{ as } \delta\to 0$ holds. However, this condition does not need to be assumed for the method in this paper. Instead it can be proved by combining (\ref{dp})-(\ref{nonnegative}) and Proposition \ref{prop:residual}.
\end{remark}

Now we are in a position to present the error estimates. To this end, assume there exists $\psi^\dag\in \textbf{Q}_{\Gamma_u}$ such that the following source condition about the exact solution holds:
\begin{equation}
	\label{sourcecondition}
	(K^*K)^\nu \psi^\dag = \phi^\dag, \ \nu>0.
\end{equation}
In the case $\nu=1$, the source condition (\ref{sourcecondition}) means there exists $\psi^\dag\in \textbf{Q}_{\Gamma_u}$ such that $\phi^\dag=w_2^\dag+i\, w_1^\dag\vert_{\Gamma_u}$ with $\textbf{w}^\dag=w_1^\dag+i\,w_2^\dag\in \textbf{V}$ being the weak solution of the problem
\begin{eqnarray*}
	\left\{
	\begin{array}{ll}
		- \Delta \textbf{w}^\dag = u_2^\dag & \textrm{in}\ \Omega,\\
		\partial_n \textbf{w}^\dag + i\,\textbf{w}^\dag = 0 & \textrm{on}\ \Gamma,
	\end{array} \right.
\end{eqnarray*}
and $\textbf{u}^\dag=u_1^\dag+i u_2^\dag\in \textbf{V}$ being the weak solution of the problem
\begin{eqnarray*}
	\left\{
	\begin{array}{ll}
		- \Delta \textbf{u}^\dag = 0 & \textrm{in}\ \Omega,\\
		\partial_n \textbf{u}^\dag + i\,\textbf{u}^\dag = 0 & \textrm{on}\ \Gamma _m, \\
		\partial_n \textbf{u}^\dag + i\,\textbf{u}^\dag = \psi^\dag & \textrm{on}\ \Gamma _u.
	\end{array} \right.
\end{eqnarray*}

For noise-free data $(\Phi,T)$ with $\delta=0$, we denote using $\phi_k=\phi_k^0\in \textbf{Q}_{\Gamma_u}$ the unique solutions to the problems (\ref{proj}). Then, under the assumption (\ref{sourcecondition}), and noticing the injectivity of the operator $K: \textbf{Q}_{\Gamma_u}\to Q$, by use of \cite[Equation (3.13')]{NP}, we obtain
\begin{equation}
	\|\phi_k-\phi^\dag\|_{\textbf{Q}_{\Gamma_u}}\leq C_1 \|\psi^\dag\|_{\textbf{Q}_{\Gamma_u}}\min_{0\leq j\leq k}\frac{\sigma_{j+1}^{2\nu}}{(k-j+1)^{2\nu}},\label{errorestimate1}
\end{equation}
where $C_1>0$ is a constant and $\sigma_{j}$ represents the $j$th singular value of $K$. From Proposition \ref{prop:K}(iii), the singular values $\{\sigma_j\}_{j=1}^{+\infty}$ decay to zero in an exponential way which indicates there exist two constants $C_2>0, 0<\rho<1$ such that
\[
\sigma_j\leq C_2 \rho^j.
\]
Then the error estimate (\ref{errorestimate1}) reduces further, to
\[
\|\phi_k-\phi^\dag\|_{\textbf{Q}_{\Gamma_u}}\leq C_1 \|\psi^\dag\|_{\textbf{Q}_{\Gamma_u}}  M(k,\rho)^{2\nu},
\]
where
\[
M(k,\rho)=\min_{0\leq j\leq k} \frac{\rho^{j+1}}{k-j+1}\leq\min_{0\leq j\leq k}\rho^{j+1}= \rho^{k+1}\to 0, \quad k\to \infty.
\]

\begin{lemma}\label{lemmaParameter}
	Let the sets of parameters $\{\gamma_k, v_k, \beta_k, p_k\}$ and $\{\gamma_k^\delta, v_k^\delta, \beta_k^\delta, p_k^\delta\}$ be generated by \textbf{Algorithm 1} applied to noise-free data $f$ and noisy data $f^\delta$, respectively. There exists a positive constant $C_3 > 0, \theta\in(0,1/2)$ such that
	\begin{equation}
		\label{smallGammaBeta}
		\min \left( \gamma_k, \beta_k \right) \geq  C_3 k^{-\theta},\ k = 1,2,\cdots.
	\end{equation}
	Then, a constant $C_4>0$ exists such that
	\begin{equation}
		\label{lemmaParameterIneq}
		\vert \gamma_{k}^\delta - \gamma_{k} \vert + \| v_{k}^\delta - v_{k} \|_{0,\Omega} + \vert \beta_{k}^\delta - \beta_{k} \vert + \| p_{k}^\delta - p_{k} \|_{\textbf{Q}_{\Gamma_u}} \leq C_4 k! \delta.
	\end{equation}
\end{lemma}

The proof is technical, and is, therefore, detailed in Appendix A.5. The assumption (\ref{smallGammaBeta}) can be empirically verified on the left of Fig. \ref{Fig: assumppf}.
\begin{figure}[H]
	\subfigure[]{
		\begin{minipage}[t]{0.47\textwidth}
			\centering
			\includegraphics[height=0.24 \textheight]{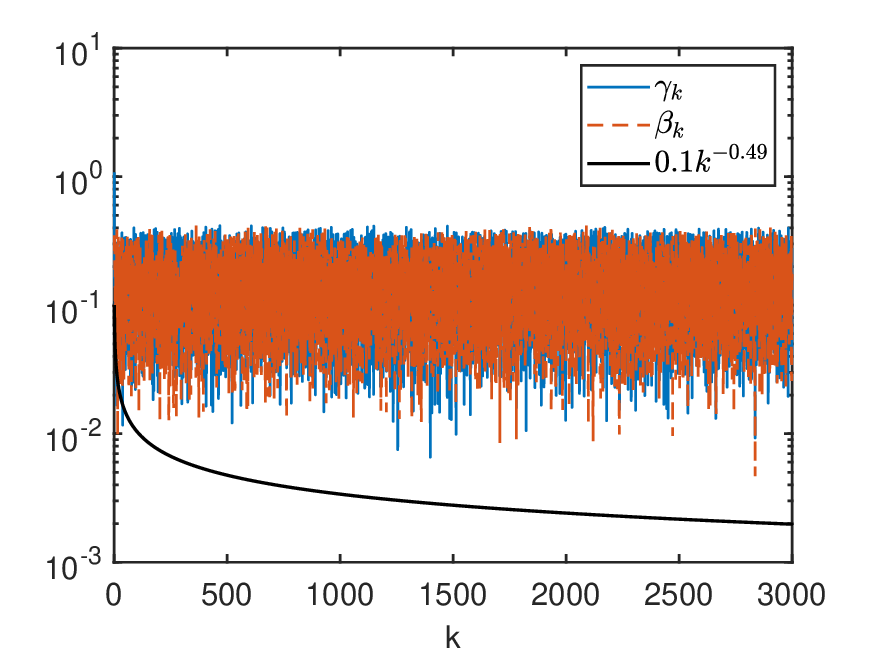}
		\end{minipage}
	}
	\subfigure[]{
		\begin{minipage}[t]{0.45\textwidth}
			\centering
			\includegraphics[height=0.24 \textheight]{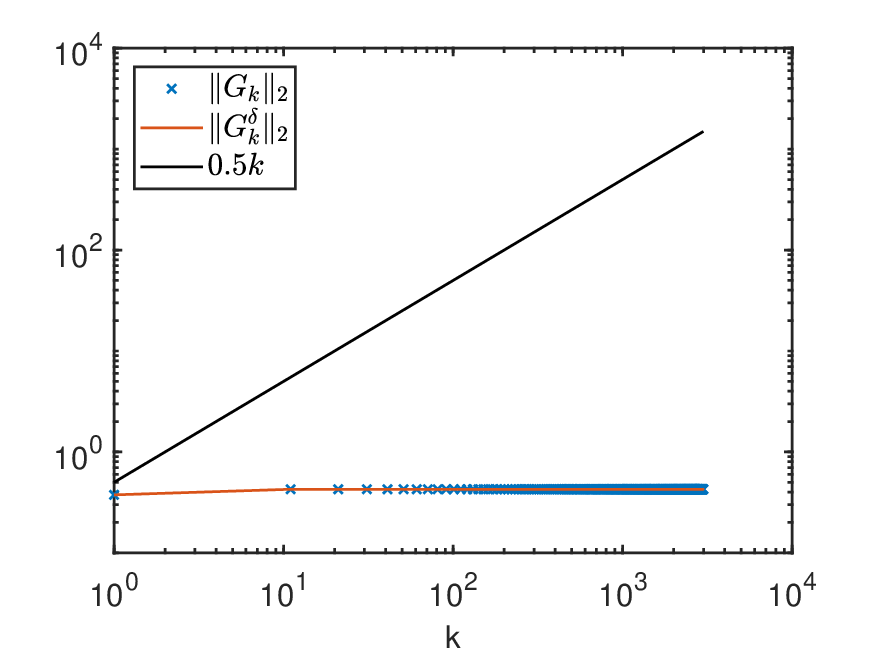}
		\end{minipage}
	}
	\caption{The numerical behaviors of quantities $\gamma_k, \beta_k, 0.1k^{-0.49}, \|G_k\|_2, \|G_k^\delta\|_2, 0.5k$ for different iterative steps $k$ in Example 1 in Section 4.2.1. (a) Numerical verification of inequality (\ref{smallGammaBeta}) with $\theta=0.49$ and $C_3=0.1$. (b) Numerical verification of Assumption \ref{assumpG} with $C_G=0.5$.}\label{Fig: assumppf}
\end{figure}

\begin{assumption}\label{assumpG}
	For all $k$, $\textrm{rank}(G_k)=\textrm{rank}(G_k^\delta)$, and there exists a constant $C_G$ such that $\max\{\|G_k\|_2,\|G_k^\delta\|_2\}\leq C_G k$.
\end{assumption}

As can be seen from the right side of Fig. \ref{Fig: assumppf}, Assumption \ref{assumpG} is quite weak in practice.
\begin{lemma}\label{ginverse}
	Let Assumption \ref{assumpG} hold, and the matrices $G_k^\delta$ as well as $G_k$ be determined by \textbf{Algorithm 1} applied to $f^\delta$ and $f$, respectively. Then, there is a constant $C_4$ satisfying
	\begin{equation}\label{matrixerror}
		\|(G_k^\delta)^\dagger - G_k^\dagger\|_2 \leq 2\sqrt{2}C^2_G C_4 (k+3)!\delta.
	\end{equation}
\end{lemma}

\begin{assumption}\label{assump_mu}
	There exist three constants $C_a, C_b>0, p>1$ and $\sigma\geq p/(2\nu\ln(\rho^{-1}))$ such that
	\begin{equation}\label{muEstimate}
		C_a e^{-e^{k/\sigma}} \leq \vert\bar{\mu}_{k}\vert \leq C_b e^{-p k \ln k}.
	\end{equation}
\end{assumption}

The lower bound in (\ref{muEstimate}) is an extremely weak assumption of the behavior of residual error $\|K\phi_k^\delta-f^\delta\|_{0,\Omega}=\vert\bar{\mu}_{k+1}\vert$.

\begin{lemma}\label{kRate}
	Under Assumption \ref{assump_mu}, there exists a constant $C_5$ such that
	\begin{equation}\label{kinequality}
		(k+3)!\leq C_5 \vert\bar{\mu}_{k}\vert^{-1} \ln^{-p} (\vert\bar{\mu}_{k}\vert^{-1}), \quad \rho^{2\nu(k+1)} \leq C_5 \ln^{-p} (\vert\bar{\mu}_{k+1}\vert^{-1}).
	\end{equation}
\end{lemma}

The proofs of Lemmas \ref{ginverse} and \ref{kRate} can be found in Appendix A.6 and Appendix A.7 respectively. Finally, based on Lemmas \ref{lemmaParameter}-\ref{kRate}, we show a convergence rate for the regularized solution $\phi_{k(\delta)}^\delta$.

\begin{theorem}
	Let Assumptions \ref{assump:2}-\ref{assump_mu} and the source condition (\ref{sourcecondition}) hold. Let $\phi_{k(\delta)}^\delta$ be constructed by the generalized GKB method with stopping rule (\ref{dp}). Then,
	\begin{eqnarray}\label{convergencerate}
		\|\phi_{k(\delta)}^\delta - \phi^\dag\|_{\textbf{Q}_{\Gamma_u}} = \mathcal{O}(\ln^{-p}(\delta^{-1})).\label{errorestimate2}
	\end{eqnarray}
\end{theorem}
\begin{proof}
	By the triangle inequality
	\begin{align*}
		\|\phi_{k(\delta)}^\delta - \phi^\dag\|_{\textbf{Q}_{\Gamma_u}} & \leq \|\phi_{k(\delta)}^\delta-\phi_k\|_{\textbf{Q}_{\Gamma_u}}+\|\phi_k-\phi^\dag\|_{\textbf{Q}_{\Gamma_u}}\\
		& \leq \|\phi_{k(\delta)}^\delta-\phi_k\|_{\textbf{Q}_{\Gamma_u}} + C_1 \|\psi^\dag\|_{\textbf{Q}_{\Gamma_u}}\rho^{2\nu(k+1)}.
	\end{align*}
	We estimate the first term on the right-hand side. By applying the generalized GKB method to solve (\ref{operator}) and (\ref{operatoreqnnoise}) separately, we obtain
	\[
	\phi_{k}^\delta = P_k^\delta \star (G_k^\delta)^\dagger\|f^\delta\|_{0,\Omega}e_1,\quad \phi_k = P_k \star G_k^\dagger \|f\|_{0,\Omega}e_1,
	\]
	where $P_k^\delta=[p_1^\delta,p_2^\delta,\cdots,p_k^\delta]$ and $P_k=[p_1,p_2,\cdots,p_k]$, obtained from the GKB process, and
	\[
	G_k^\delta =
	\left(
	\begin{array}{cccc}
		\beta_1^\delta  &                &                 &                 \\
		\gamma_2^\delta & \beta_2^\delta &                 &                 \\	
		                & \ddots         & \ddots          &                 \\
		                &                & \gamma_k^\delta & \beta_k^\delta  \\
		                &                &                 & \gamma_{k+1}^\delta
	\end{array} \right)_{(k+1)\times k},
	G_k =
	\left(
	\begin{array}{cccc}
		\beta_1  &         &          &            \\
		\gamma_2 & \beta_2 &          &             \\	
		         & \ddots  & \ddots   &             \\
		         &         & \gamma_k & \beta_k      \\
		         &         &          & \gamma_{k+1}
	\end{array} \right)_{(k+1)\times k}.
	\]
	Then, it follows from $\|P_k^\delta\|_{\textbf{Q}_{\Gamma_u}} = 1$, Lemmas \ref{lemmaParameter} and \ref{ginverse} that
	\begin{align*}
		&\|\phi_{k}^\delta-\phi_k\|_{\textbf{Q}_{\Gamma_u}} = \|P_k^\delta\star (G_k^\delta)^\dagger\|f^\delta\|_{0,\Omega}e_1 - P_k\star (G_k)^\dagger\|f\|_{0,\Omega}e_1\|_{\textbf{Q}_{\Gamma_u}}\\
		& = \|P_k^\delta\star(G_k^\delta)^\dagger(\|f^\delta\|_{0,\Omega} - \|f\|_{0,\Omega})e_1 + P_k^\delta\star ((G_k^\delta)^\dagger - G_k^\dagger)\|f\|_{0,\Omega}e_1\\
		& \quad + (P_k^\delta - P_k)\star G_k^\dagger\|f\|_{0,\Omega}e_1\|_{\textbf{Q}_{\Gamma_u}} \\
		& \leq \|(G_k^\delta)^\dagger\|_2 C_f\delta + \|(G_k^\delta)^\dagger - G_k^\dagger \|_2\|f\|_{0,\Omega} + \|P_k^\delta-P_k\|_{\textbf{Q}_{\Gamma_u}} \|G_k^\dagger\|_2\|f\|_{0,\Omega}\\
		& \leq \left( C_fC_G k +2\sqrt{2}C^2_G C_4 (k+3)! \|f\|_{0,\Omega} + C_4 k! C_G k \|f\|_{0,\Omega} \right)\delta\\
		& \leq  C_6 (k+3)!\delta,
	\end{align*}
	where $C_6 = C_fC_G  +2\sqrt{2}C^2_G C_4 \|f\|_{0,\Omega} + C_4  C_G \|f\|_{0,\Omega}$. According to Lemma \ref{kRate} and the stopping rule in (\ref{dp}), we deduce that
	\begin{align*}
		& \|\phi_{k(\delta)}^\delta - \phi^\dag\|_{\textbf{Q}_{\Gamma_u}}
		\leq C_6 (k+3)!\delta + C_1 \|\psi^\dag\|_{\textbf{Q}_{\Gamma_u}} \rho^{2\nu(k+1)} \\
		& \leq C_5C_6 \tau^{-1} \ln^{-p}(\tau^{-1}\delta^{-1}) + C_1 C_5 \|\psi^\dag\|_{\textbf{Q}_{\Gamma_u}} \ln^{-p}(\tau^{-1}\delta^{-1}),
	\end{align*}
	which yields the required estimate (\ref{convergencerate}).\qed
\end{proof}

\section{Numerical simulation}\label{sec:numer}

\subsection{Discretization with finite element methods}\label{subsec:fem}

In this subsection, finite element methods \cite{BGF2022} are used to solve BVPs such as (\ref{databvp}), (\ref{forwardbvp}) and (\ref{adjointeqn}). We note that they can also be solved by the boundary element method \cite{VLD2021,MEIL2000} and the boundary integral equation method \cite{AG1980,S2017}. The numerical implementation mainly involves the computation $\|\phi\|_{\textbf{Q}_{\Gamma_u}}$, $\|g\|_{0,\Omega}, K\phi$ and $K^*g$ for some $\phi\in \textbf{Q}_{\Gamma_u}$ and $g\in Q$. Standard conforming linear finite element methods are applied for this purpose. Specifically, let $\{\mathcal{T}_h\}$ be a regular family of finite element partitions of $\overline{\Omega}$, and define the real linear finite element space
\[
V^h = \{v\in \mathcal{C}(\overline{\Omega})\vert\ v\textrm{ is linear in }T\quad \forall\ T\in \mathcal{T}_h\}.
\]
Set $\textbf{V}^h= V^h\oplus i\, V^h$ as the complex version of $V^h$.

Recall that for any $(\Phi^\delta,T^\delta)\in Q_m$, $\phi=\varphi+i\,t\in \textbf{Q}_{\Gamma_u}$, $g\in Q$, we have $f^\delta=-\hat{u}_2$ with $\hat{\textbf{u}}=\hat{u}_1+i\,\hat{u}_2\in \textbf{V}$ solving
\begin{equation}
	a(\hat{\textbf{u}},\textbf{v})= F(\Phi^\delta,T^\delta,0,0;\textbf{v}) \quad\forall\,\textbf{v}\in \textbf{V}, \label{weak1}
\end{equation}
$K\phi=\tilde{u}_2$ with $\tilde{\textbf{u}}=\tilde{u}_1+i\,\tilde{u}_2\in \textbf{V}$ solving
\begin{equation}
	a(\tilde{\textbf{u}},\textbf{v})= F(0,0,\varphi,t;\textbf{v}) \quad\forall\,\textbf{v}\in \textbf{V}, \label{weak2}
\end{equation}
and $K^*g=w_2+i\, w_1\vert_{\Gamma_u}$ with $\textbf{w}=w_1+i\,w_2\in \textbf{V}$ solving
\begin{equation}
	a(\textbf{w},\textbf{v})= (g, \textbf{v})_{0,\Omega} \quad\forall\,\textbf{v}\in \textbf{V}. \label{weak3}
\end{equation}
Here $a(\cdot,\cdot)$ and $F$ are defined in (\ref{auv}) and (\ref{fv}), respectively.

Let $n$ denote the number of nodes of triangulation $\mathcal{T}_h$, and $\{\varphi_l^h\}_{l=1}^n\subset V^h$ denote the nodal basis functions of $V^h$ associated with the grid points $\{x_l\}_{l=1}^n$. Define
\[
\begin{array}{llll}
	A=(a_{ls})_{n\times n},& a_{ls}=\int_{\Omega}\nabla \varphi_s^h\cdot\nabla\varphi_l^h\,\textrm{d}x, \\[0.3em]
	M=(m_{ls})_{n\times n},& m_{ls}=\int_{\Omega} \varphi_s^h\,\varphi_l^h\,\textrm{d}x  \\[0.3em]
	C=(c_{ls})_{n\times n},& c_{ls}=\int_{\Gamma} \varphi_s^h\,\varphi_l^h\,\textrm{d}s,\\ [0.3em]
	C_m=(c^m_{ls})_{n\times n},& c^m_{ls}=\int_{\Gamma_m} \varphi_s^h\,\varphi_l^h\,\textrm{d}s,\\ [0.3em]
	C_u=(c^u_{ls})_{n\times n},& c^u_{ls}=\int_{\Gamma_u} \varphi_s^h\,\varphi_l^h\,\textrm{d}s,\\ [0.3em]
	l, s=1, 2, \cdot\cdot\cdot, n.  &
\end{array}
\]
Then, without going into details, by applying the finite element method, (\ref{weak1})--(\ref{weak3}) are reduced to three linear systems as follows:

\begin{equation}
	\left \{
	\begin{array}{ll}
		A\,\hat{u}_1-C\,\hat{u}_2= C_m \Phi^\delta,\\[0.5em]
		C\,\hat{u}_1+A\,\hat{u}_2=C_m T^\delta,
	\end{array}
	\right.\label{system1}
\end{equation}
\begin{equation}
	\left \{
	\begin{array}{ll}
		A\,\tilde{u}_1-C\,\tilde{u}_2= C_u \varphi,\\[0.5em]
		C\,\tilde{u}_1+A\,\tilde{u}_2=C_u t,
	\end{array}
	\right.\label{system2}
\end{equation}
and
\begin{equation}
	\left \{
	\begin{array}{l}
		A\,w_1-C\,w_2= M\,g,\\[0.5em]
		C\,w_1+A\,w_2=0.
	\end{array}
	\right.\label{system3}
\end{equation}
What we wish to point out, is that for the sake of simplicity, the same symbols are used for both functions and their vector expansions in $\mathbb{R}^n$, associated with nodal basis functions $\{\varphi_l^h\}_{l=1}^n$. No confusion is expected with the help of context.

As a result, in \textbf{Algorithm 1}, Eq. (\ref{system1}) is solved for $f^\delta$, Eq. (\ref{system2}) is solved for the operation of $K$, and Eq. (\ref{system3}) is solved for the operation of $K^*$. Moreover, for any $\phi=\varphi+i\,t\in \textbf{Q}_{\Gamma_u}, g\in Q$,
\begin{equation}
	\|\phi\|_{\textbf{Q}_{\Gamma_u}}\approx\left(\varphi^T C_u \varphi +t^T C_u t\right)^{1/2},
	\|g\|_{0,\Omega}\approx\left(g^T M g\right)^{1/2}.\label{norms}
\end{equation}

\subsection{Numerical results}\label{subsec:numresult}

In this part, we focus on presenting some numerical results for the Cauchy problem. \textbf{Algorithm 1}, together with the discrepancy principle (\ref{dp}) and finite element approximations (\ref{system1})--(\ref{norms}), are implemented for the reconstruction of $(\varphi^\dag, t^\dag)$. Let $(\varphi^h_{k(\delta)},t^h_{k(\delta)})$ denote the obtained approximate solutions, where $h$ is the mesh size of the finite element partition of $\mathcal{T}_h$. In order to better estimate the accuracy of the solutions and compare the accuracy of the different methods, we define the following relative errors in approximate solutions:
\[
\textrm{Err}_\varphi = \frac{\|\varphi^h_{k(\delta)}- \varphi^\dag\|_{0,\Gamma_u}}{\|\phi^\dag\|_{0,\Gamma_u}},\quad
\textrm{Err}_t = \frac{\|t^h_{k(\delta)} - t^\dag\|_{0,\Gamma_u}}{\|t^\dag\|_{0,\Gamma_u}}.
\]

For comparison with the existing work, we take examples from \cite{AN,CGH16}. Specifically, let $\Omega \subset \mathbb{R}^2$ be a ring with inner radius $r_1$ and external radius $r_2$. Take the external circle as $\Gamma_m$ and the inner circle as $\Gamma_u$. For simulation, in each example, we are given a true state $u^\dag$. Then the noise-free data and true solutions are computed using
\[
\Phi = \partial_n u^\dag\vert_{\Gamma_m},\ T = u^\dag\vert_{\Gamma_m}, \
\varphi = \partial_n u^\dag\vert_{\Gamma_u},\ t = u^\dag\vert_{\Gamma_u}.
\]
With a given relative noise level $\delta^\prime$, the noisy data $(\Phi^\delta,T^\delta)$ is produced by
\begin{eqnarray*}
	\Phi^\delta(x,y) = (1 + \delta^\prime \cdot 2 \cdot (\textrm{rand}(x,y) - 0.5))\Phi(x,y),\\
	T^\delta(x,y) = (1 + \delta^\prime \cdot 2 \cdot (\textrm{rand}(x,y) - 0.5))T(x,y),
\end{eqnarray*}
with $(x,y)\in \Gamma_m$, where rand$(x,y)$ returns a random number of uniform distributions on $[0,1]$. The noise level $\delta$ is computed as
\[
\delta = \|f^\delta-f\|_{0,\Omega}.
\]
Then, the approximate solutions $(\varphi^h_{k(\delta)},t^h_{k(\delta)})$ on $\Gamma_u$ are reconstructed from $(\Phi^\delta,T^\delta)$ on $\Gamma_m$. All experiments are implemented on a finite element mesh with 1796 nodes, 3392 elements and mesh size $h=0.1375$. Moreover, in all examples, we set $r_1 = 1, r_2 = 2$ and  the maximal iterative number $N_{\text{max}}=1000$.

We compare the method proposed in this paper with the Landweber method and the conjugate gradient\,(CG) method. With the CCBM reformulation, both the Landweber and CG methods can be readily applied in continuous context. Otherwise, for other frameworks such as those $K$ listed in the Appendices, fractional-order spaces have to be encountered and the numerical simulations are unfriendly. For convenience of statement, we abbreviate the CCBM-based Landweber method, CG method and GKB method, as CCBM-L, CCBM-CG and CCBM-GKB, respectively.

\subsubsection{Example 1.}

In the first example, set $u^\dag = e^x \cos y$ in $\Omega$. Then on $\Gamma_m$, $\Phi = e^x(x\cos y - y\sin y)/2, T=e^x \cos y$. The true solutions are $\varphi^\dag = -e^x(x\cos y - y\sin y)$, $t^\dag= e^x\cos y$.

For $\delta^\prime=0.01, 0.05, 0.1$, we compute the condition numbers of $G_j$ for different iterative steps $j$, and show them in Table \ref{tab:11}. As indicated by Table \ref{tab:11}, compared with the exponential decaying behavior of the singular values of $K$, the condition numbers of $G_j$ increase relatively gently. Therefore, when $j$ is not large, we do not introduce any additional regularization strategy for solving the reduced optimal problem (\ref{opteqn}). This is also true for the examples below. However, for the conciseness of this paper, similar results are omitted.

\begin{table}[H]
	\caption{The dependence of the condition numbers of $G_j$ on $j$.}\label{tab:11}
	\begin{center}			
		\begin{tabular}{|c|c|c|c|} \hline
				$j$    & $\delta^\prime=0.01$  & $\delta^\prime=0.05$ & $\delta^\prime=0.1$\\  \hline 
				1      &        1              &        1	          &       1             \\
				2      &        1.5829         &        1.5905        &       1.5987        \\
				3      &        2.0772         &        2.2044	      &       2.4479        \\
				4      &        3.4388         &        3.4366        &       3.4316        \\
				5      &        3.6324         &        3.6380	      &       3.6453        \\
				6      &        5.0179         &        5.0884        &       5.1776        \\
				7      &        5.2368         &        5.3078	      &       5.3988        \\
				8      &        7.3418         &        7.4636	      &       7.6163        \\
				9      &        8.0695         &        8.1578	      &       8.2741        \\
				10     &        16.5797        &        17.3378       &       17.6157       \\  \hline
		\end{tabular}
	\end{center}
\end{table}

\begin{table}[H]
	\caption{Comparison of different methods for different $\delta^\prime$\,(Example 1).}\label{tab:exam12}
	\begin{center}
		\begin{tabular}{|c|c|c|c|c|} \hline
				\multicolumn{2}{|c|}{$\delta^\prime$}        & 0.01        & 0.05        & 0.1         \\ \hline
				\multirow{3}{*}{Err$_\varphi$}   & CCBM-L    & 1.4731e-1   & 2.5222e-1   & 2.6929e-1   \\
				                                 & CCBM-CG   & 1.1946e-1   & 2.6669e-1   & 2.5406e-1   \\
				                                 & CCBM-GKB  & 1.1870e-1   & 2.6669e-1   & 2.5406e-1   \\ \hline
				\multirow{3}{*}{Err$_t$}         & CCBM-L    & 4.8446e-2   & 1.2540e-1   & 1.8508e-1   \\
				                                 & CCBM-CG   & 2.8307e-2   & 1.2178e-1   & 1.9434e-1   \\
				                                 & CCBM-GKB  & 2.7818e-2   & 1.2178e-1   & 1.9434e-1   \\ \hline
				\multirow{3}{*}{$k(\delta)$}     & CCBM-L    & 172         & 19 	     & 12          \\
				                                 & CCBM-CG   & 10          & 5	         & 4           \\
				                                 & CCBM-GKB  & 10          & 5           & 4           \\ \hline
			\end{tabular}
	\end{center}
\end{table}

For different relative noise levels $\delta^\prime$, the CCBM-L, CCBM-CG and CCBM-GKB methods are applied to the Cauchy problem specified here. The relative errors in solutions and the corresponding iterative steps are displayed in Table \ref{tab:exam12}, whence we conclude that all three methods yield satisfactory numerical solutions, with comparable levels of accuracy. By comparison, both the CCBM-CG and CCBM-GKB methods exhibit acceleration behavior. Moreover, the CCBM-GKB method affords slightly better solutions than the CCBM-CG method, even as both methods follow the same iterative steps. The approximate solutions derived from \textbf{Algorithm 1} are shown in Fig. \ref{Fig:11}. From Table \ref{tab:exam12} and Fig. \ref{Fig:11}, the larger the value of $\delta^\prime$ is, the worse the solution accuracy, and the earlier the iteration needs to be stopped.

\begin{figure}[H]
	\centering
	\begin{minipage}[t]{0.48\textwidth}
		\centering
		\includegraphics[height=0.39 \textheight]{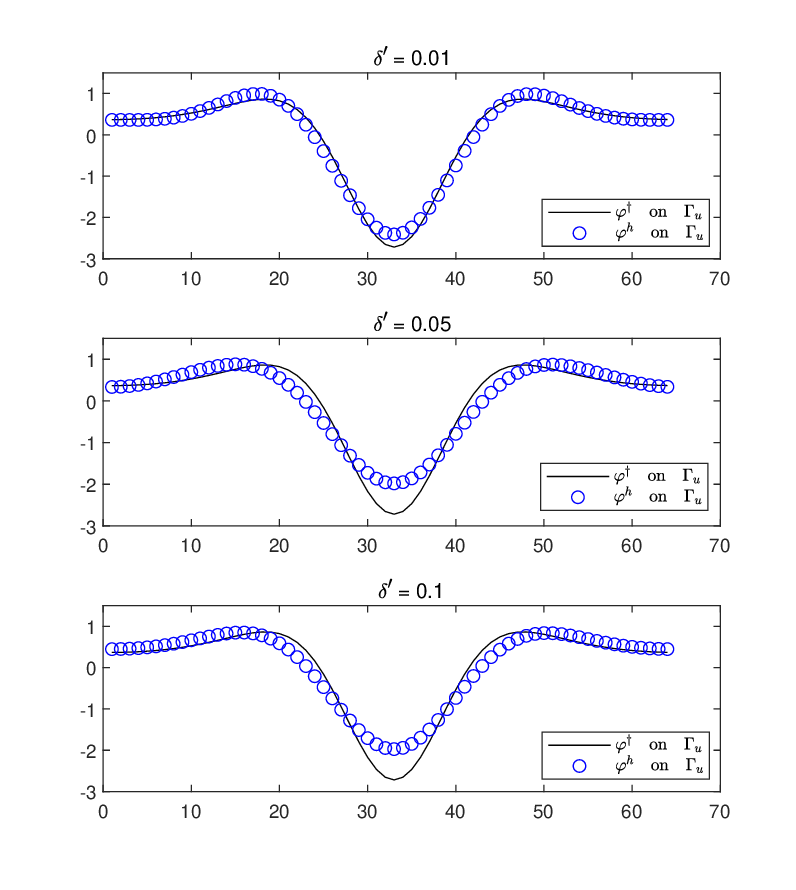}
	\end{minipage}
	\begin{minipage}[t]{0.48\textwidth}
		\centering
		\includegraphics[height=0.39 \textheight]{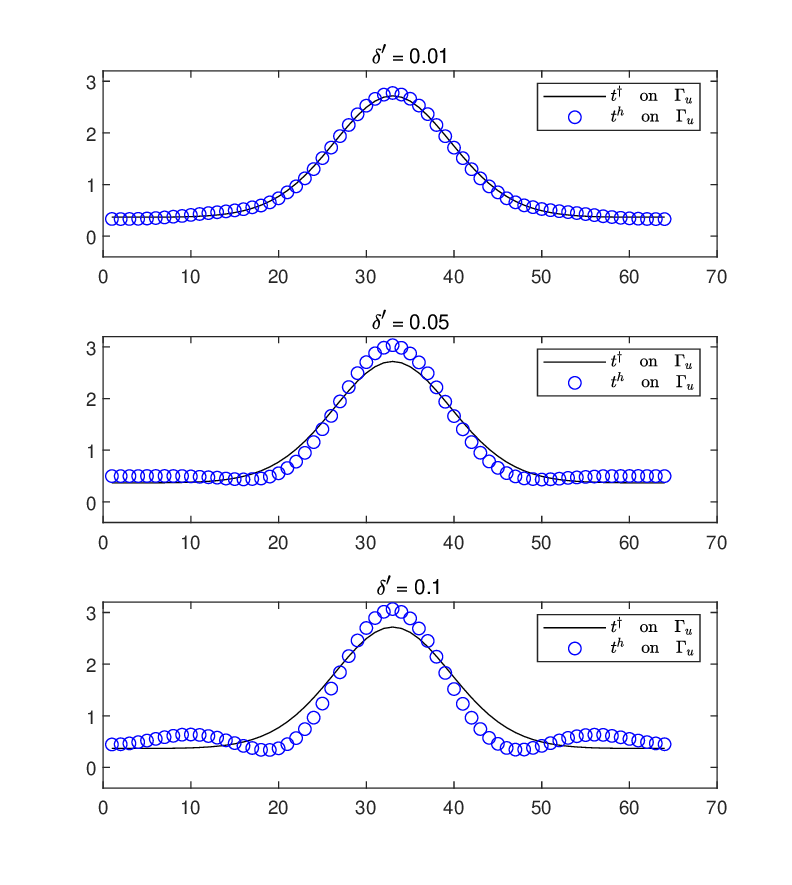}
	\end{minipage}
	\caption{$\varphi^h_{k(\delta)}$ (left) and $t^h_{k(\delta)}$ (right) for different $\delta^\prime$\,(Example 1).}\label{Fig:11}
\end{figure}

Finally, we discuss the case of discontinuous Dirichlet data on $\Gamma_u$. To this end, given the exact expression
\[
t^\dagger = \left\{
\begin{array}{ll}
         e^x\cos y & y=\sqrt{1-x^2},     \\
         0         & y=-\sqrt{1-x^2}, \\
\end{array} \right. \textrm{ on }\ \Gamma_u,
\]
which is discontinuous at the point $(1,0)$, and $\textrm{on}\ \Gamma_m,\ \Phi = e^x(x\cos y - y\sin y)/2$. $T$ can be obtained by solving the BVP which uses $t^\dagger$ as Dirichlet boundary data on $\Gamma_u$ and $\Phi$ as Neumann boundary data on $\Gamma_m$. 

We apply the CCBM-L, CCBM-CG and CCBM-GKB to compute approximate solution $(\varphi_k^h,t_k^h)$ from the reconstructed noisy Cauchy data $(\Phi^\delta,T^\delta)$. Note that there is no precise expression for $\varphi$ here and then it is difficult to calculate its relative error. Thus, we just present the relative error of $t$ and iterative number $k$ in Table \ref{tab:exam13}. We conclude from Table \ref{tab:exam13} that compared to CCBM-L, CCBM-CG and CCBM-GKB have a similar acceleration effect. In addition, compared with Table \ref{tab:exam12}, the relative error of $t$ in Table \ref{tab:exam13} is relatively large.

It can be observed from Fig. \ref{Fig:12} that the discontinuous point contribute to a major component of the large relative error Err$_t$. One possible reason is that the reconstruction of the base functions $\{v_i\}_i$ and $\{p_i\}_i$ in \textbf{Algorithm 1} implies implicitly $L^2$ regularization. Therefore, the reconstruction for the smooth solution by using the proposed algorithm is better than the case for the discontinuous solution. 

\begin{table}[H]
	\caption{The relative error of $t$ and the corresponding iterative number for different $\delta^\prime$.}\label{tab:exam13}
	\begin{center}
		\begin{tabular}{|c|c|c|c|c|} \hline
			\multicolumn{2}{|c|}{$\delta^\prime$}        & 0.01              & 0.05        & 0.1 \\ \hline
			\multirow{3}{*}{Err$_t$}         & CCBM-L    & 3.5102e-1         & 4.4433e-1   & 5.6206e-1   \\
			                                 & CCBM-CG   & 3.0849e-1         & 3.6777e-1   & 5.5603e-1   \\
			                                 & CCBM-GKB  & 3.0845e-1         & 3.6828e-1   & 5.5603e-1  \\ \hline
			\multirow{3}{*}{$k(\delta)$}     & CCBM-L    & $N_{\text{max}}$  & 191         & 54      \\
			                                 & CCBM-CG   &   17              & 10          & 8      \\
			                                 & CCBM-GKB  &   17              & 10          & 8        \\ \hline
		\end{tabular}
	\end{center}
\end{table}

\begin{figure}[H]
	\begin{center}
		\includegraphics[height=0.4 \textheight]{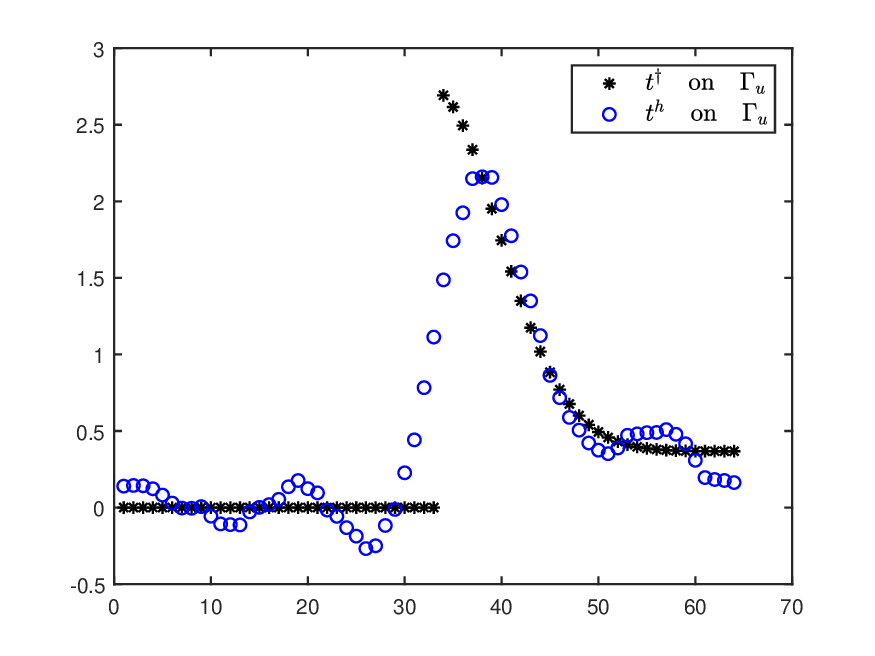}
		\caption{$t^h_{k(\delta)}$ for $\delta^\prime=0.01$.}\label{Fig:12}
	\end{center}
\end{figure}

\subsubsection{Example 2.}

In this example, we set $u^\dag = x^2-y^2$ in $\Omega$. Then on $\Gamma_m$, $\Phi = x^2-y^2, T=x^2-y^2$. The true solutions are $\varphi^\dag = 2y^2-2x^2$, $t^\dag = x^2-y^2$.

After the noisy Cauchy data is constructed, the three aforementioned methods are applied to the underlying Cauchy problem. The experimental results are shown in Table \ref{tab:exam21}. A similar conclusion to that in Example 1 can be drawn from Table \ref{tab:exam21}. In addition, the approximate solutions solved by \textbf{Algorithm 1} are shown in Fig. \ref{Fig:21}.

\begin{table}[H]
	\caption{ Comparison of different methods for different $\delta^\prime$\,(Example 2).}\label{tab:exam21}
	\begin{center}
		\begin{tabular}{|c|c|c|c|c|} \hline
				\multicolumn{2}{|c|}{$\delta^\prime$}             & 0.01        & 0.05           & 0.1 \\\hline
				\multirow{3}{*}{Err$_\varphi$}        & CCBM-L    & 7.1608e-3   & 1.1702e-2	     & 4.7609e-2    \\
				                                      & CCBM-CG   & 2.7581e-3   & 3.1189e-3 	 & 1.0440e-2    \\
				                                      & CCBM-GKB  & 2.7581e-3   & 3.1188e-3	     & 1.0440e-2    \\ \hline
				\multirow{3}{*}{Err$_t$}              & CCBM-L    & 1.4390e-2   & 2.3294e-2	     & 9.3819e-2    \\
				                                      & CCBM-CG   & 6.4128e-3   & 7.2509e-3 	 & 2.1435e-2    \\
				                                      & CCBM-GKB  & 6.4128e-3   & 7.2509e-3	     & 2.1434e-2    \\ \hline
				\multirow{3}{*}{$k(\delta)$}          & CCBM-L    & 136         & 113	         & 60           \\
			                                     	  & CCBM-CG   & 4           & 4	             & 4            \\
				                                      & CCBM-GKB  & 4           & 4	             & 4            \\ \hline
		\end{tabular}
	\end{center}
\end{table}

\begin{figure}[H]
	\centering
	\begin{minipage}[t]{0.48\textwidth}
		\centering
		\includegraphics[height=0.4 \textheight]{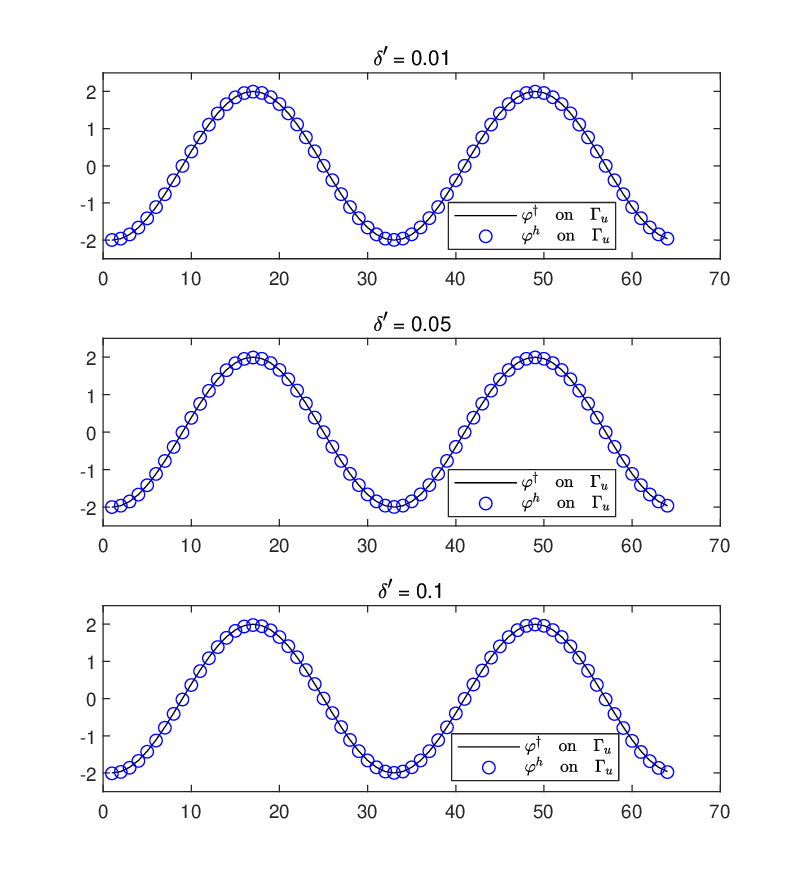}
	\end{minipage}
	\begin{minipage}[t]{0.48\textwidth}
		\centering
		\includegraphics[height=0.4 \textheight]{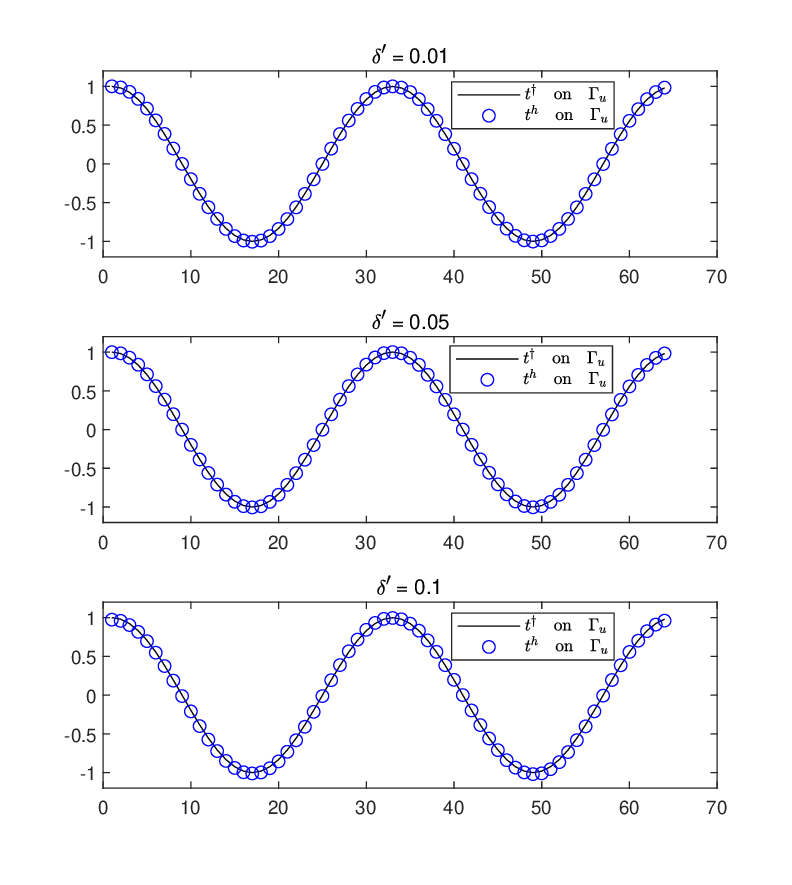}
	\end{minipage}
	\caption{$\varphi^h_{k(\delta)}$ (left) and $t^h_{k(\delta)}$ (right) for different $\delta^\prime$\,(Example 2).}\label{Fig:21}
\end{figure}

\subsubsection{Example 3.}

In contrast to Examples 1-2, in the example, we now consider a more general Cauchy problem:
\[
\left\{
\begin{array}{ll}
	- \textrm{div }(\kappa\nabla u) = 0 & \textrm{in}\ \Omega,\\
	\kappa\partial_n u = \Phi^\delta, u = T^\delta & \textrm{on}\ \Gamma _m, \\
	\kappa\partial_n u = \varphi, u = t & \textrm{on}\ \Gamma _u
\end{array} \right.
\]
with
\[
\kappa=\left[
\begin{array}{ll}
	1 & 0 \\
	0 & \zeta
\end{array} \right].
\]
This model arises in applications of orthotropic materials \cite{AN}. Let $u^\dag = e^{\sqrt{\zeta}x}\cos y$ in $\Omega$. Then on $\Gamma_m$, $\Phi = \frac{1}{2}e^{\sqrt{\zeta}x}(\sqrt{\zeta}x\cos y-\zeta y \sin y), T=e^{\sqrt{\zeta}x}\cos y$. The true solutions are $\varphi^\dag = e^{\sqrt{\zeta}x}(\zeta y \sin y-\sqrt{\zeta}x\cos y)$, $t^\dag = e^{\sqrt{\zeta}x}\cos y$.

Again, the three CCBM-based methods are applied to the specified Cauchy problem. Here we fix the relative noise level $\delta'=0.01$. The experimental results for different values of $\zeta$ are shown in Table \ref{tab:exam31}. The choice of $\zeta=1$ reduces the problem to that outlined in Example 1. We further plot the approximate solutions obtained from \textbf{Algorithm 1} in Figs \ref{Fig:31}--\ref{Fig:32}. It is clear that the values of the parameter $\zeta$ affect the accuracy in approximated solutions. In fact, we can see from these experiments that the numerical solutions obtained when $\zeta$ is far away from 1 is not as accurate as those when $\zeta$ is close to 1.

\begin{table}[H]
	\caption{Comparison of different methods for different $\zeta$\,(Example 3).}\label{tab:exam31}
	\begin{center}
		\begin{tabular}{|c|c|c|c|c|c|} \hline
			\multicolumn{2}{|c|}{$\zeta$}                & 0.01        & 0.05       & 0.1         & 2          \\ \hline
			\multirow{3}{*}{Err$_\varphi$}   & CCBM-L    & 2.3010e-1   & 4.8053e-2  & 2.5275e-2   & 1.9066e-1  \\
				                             & CCBM-CG   & 2.3411e-1   & 3.8572e-2	& 2.5629e-2   & 1.7002e-1  \\
				                             & CCBM-GKB  & 2.3411e-1   & 3.8573e-2	& 2.5626e-2   & 1.7002e-1  \\ \hline
			\multirow{3}{*}{Err$_t$}         & CCBM-L    & 1.1849e-2   & 6.2002e-3	& 7.4842e-3   & 5.3278e-2  \\
				                             & CCBM-CG   & 1.1921e-2   & 8.2983e-3 	& 6.9206e-3   & 4.0063e-2  \\
				                             & CCBM-GKB  & 1.1921e-2   & 8.2981e-3 	& 6.9195e-3   & 3.9877e-2  \\ \hline
			\multirow{3}{*}{$k(\delta)$}     & CCBM-L    & 206         & 264	    & 88          & 1039       \\
				                             & CCBM-CG   & 14          & 10	        & 12          & 16         \\
				                             & CCBM-GKB  & 14          & 10	        & 12          & 17         \\ \hline
		\end{tabular}
	\end{center}
\end{table}

\begin{figure}[H]
	\begin{center}
		\includegraphics[height=0.4 \textheight]{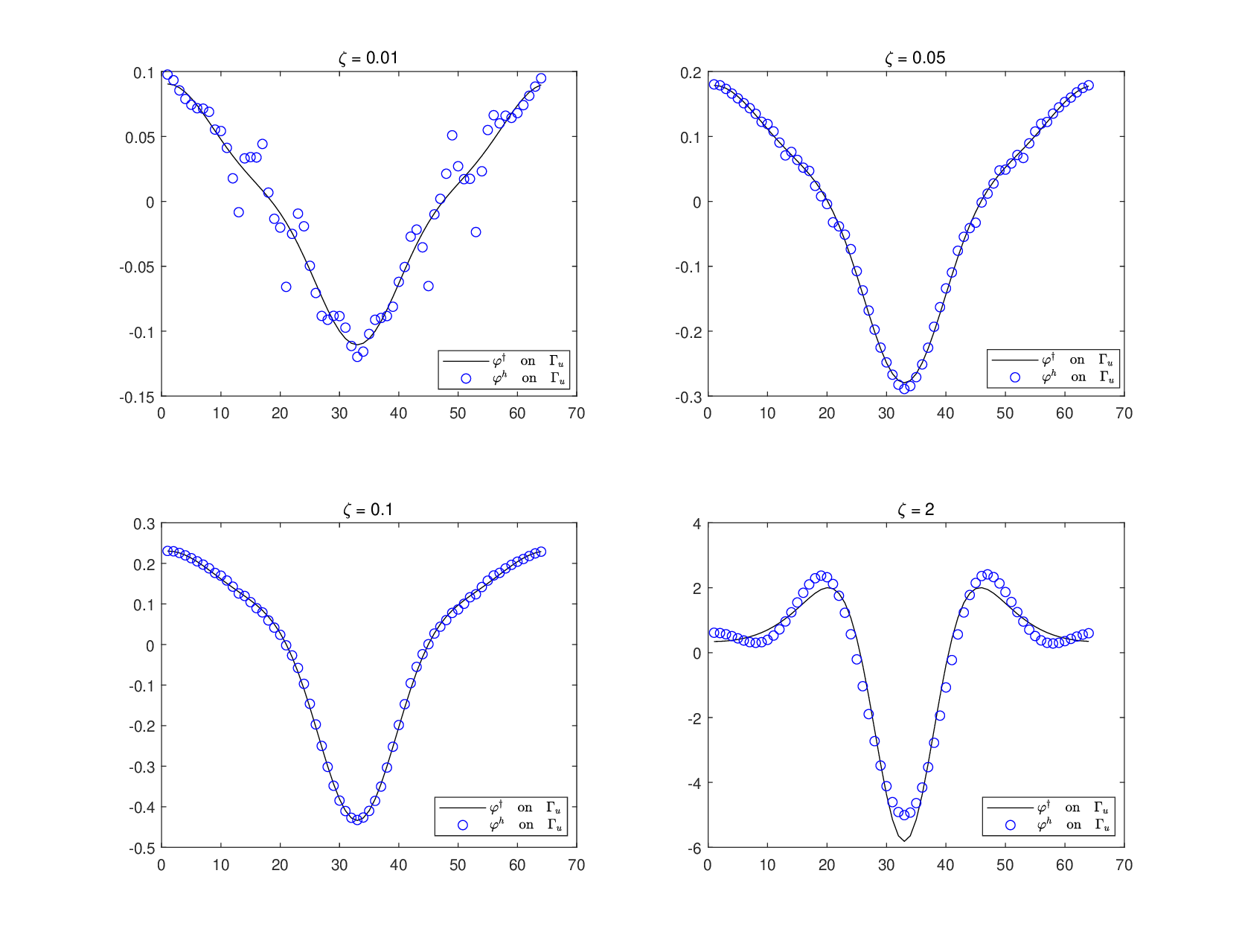}
		\caption{$\varphi^h_{k(\delta)}$ for different $\zeta$\,(Example 3).}\label{Fig:31}
	\end{center}
\end{figure}

\begin{figure}[H]
	\begin{center}
		\includegraphics[height=0.4 \textheight]{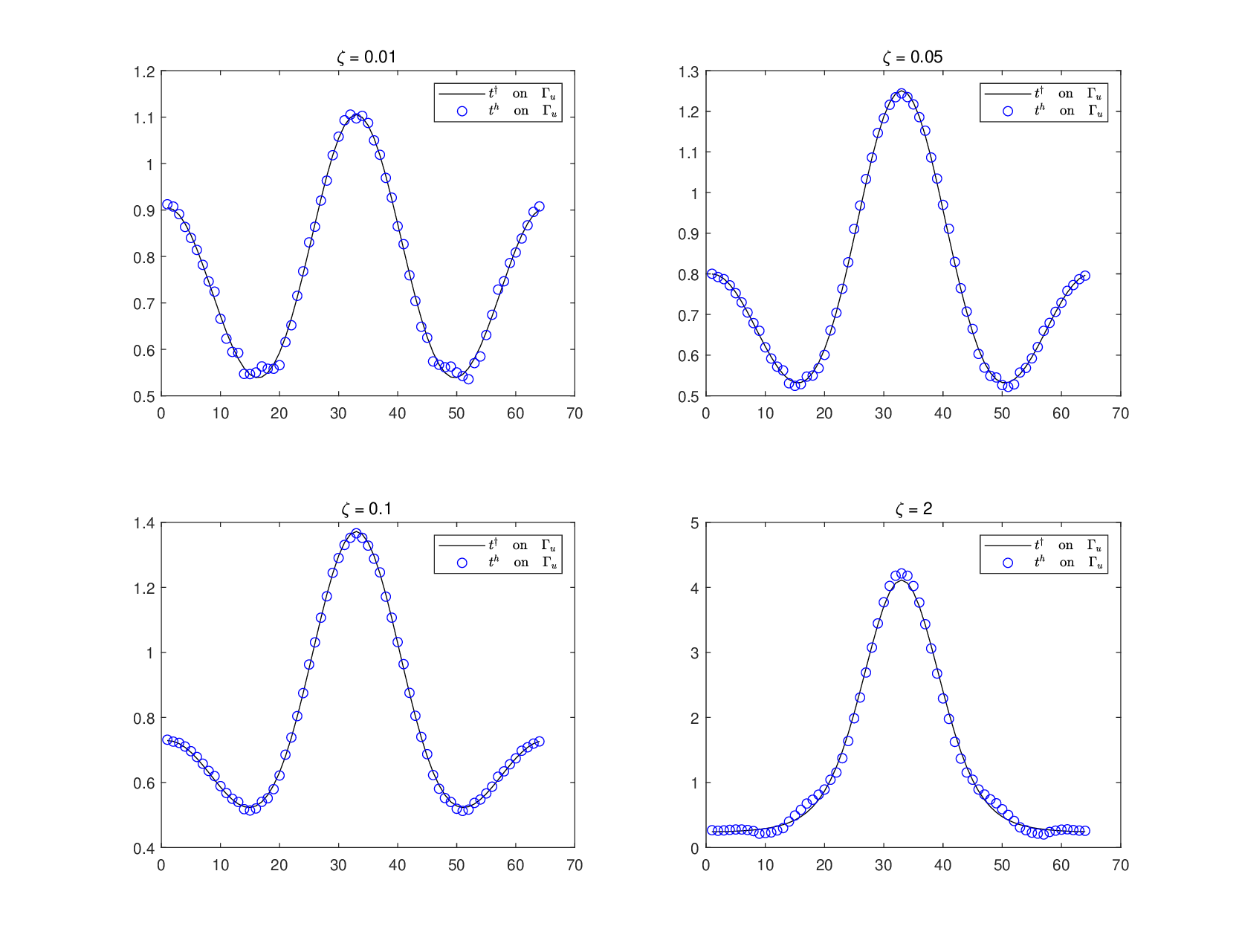}
		\caption{$t^h_{k(\delta)}$ for different $\zeta$\,(Example 3).}\label{Fig:32}
	\end{center}
\end{figure}

\subsubsection{Example 4.}

Consider once again the Cauchy problem from Example 3, but here we set 
\[
\kappa=\left[
\begin{array}{ll}
	1 & 0.3 \\
  0.3 & 0.4
\end{array} \right],
\]
and $u^\dagger=x^3/15 - x^2y + xy^2 + y^3/3$ in $\Omega$. Then on $\Gamma_m$, $\Phi=- x^3/20 - 87x^2y/100 + 3xy^2/4 + 7y^3/20, T=x^3/15 - x^2y + xy^2 + y^3/3$. The true solution are $\varphi^\dagger=x^3/10 + 87yx^2/50 - 3y^2x/2 - 7y^3/10, t^\dagger=x^3/15 - x^2y + xy^2 + y^3/3$. This model arises in applications of anisotropic materials \cite{VLD2021}.

Similarly, Table \ref{tab:exam41} presents some numerical results, which are obtained from using the three CCBM-based methods. Meanwhile, we plot the exact solution and the approximate solutions obtained from \textbf{Algorithm 1} for different noise level in Fig. \ref{Fig:41}. By comparing the numerical results in Table \ref{tab:exam41}, we can obtain conclusions similar to those in Example 1. Moreover, Table \ref{tab:exam41} and Fig. \ref{Fig:41} illustrate that \textbf{Algorithm 1} converges and the reconstruction is satisfactory.

\begin{table}[H]
	\caption{ Comparison of different methods for different $\delta^\prime$\,(Example 4).}\label{tab:exam41}
	\begin{center}
	\begin{tabular}{|c|c|c|c|c|} \hline
		\multicolumn{2}{|c|}{$\delta^\prime$}             & 0.01        & 0.05         & 0.1         \\\hline
		\multirow{3}{*}{Err$_\varphi$}        & CCBM-L    & 8.2951e-2   & 2.1503e-1	   & 2.8750e-1   \\
		                                      & CCBM-CG   & 7.0633e-2   & 1.7078e-1	   & 2.5423e-1   \\
		                                      & CCBM-GKB  & 7.0632e-2   & 1.7078e-1	   & 2.5423e-1   \\ \hline
		\multirow{3}{*}{Err$_t$}              & CCBM-L    & 7.2127e-2   & 2.8562e-1    & 4.2404e-1   \\
		                                      & CCBM-CG   & 5.9308e-2   & 2.2537e-1	   & 3.8236e-1   \\
	  	                                      & CCBM-GKB  & 5.9309e-2   & 2.2537e-1	   & 3.8236e-1   \\ \hline
		\multirow{3}{*}{$k(\delta)$}          & CCBM-L    &  219        & 35           & 9          \\
		                                      & CCBM-CG   &  11         & 8            & 4            \\
		                                      & CCBM-GKB  &  11         & 8            & 4             \\ \hline
	\end{tabular}
\end{center}
\end{table}

\begin{figure}[H]
	\centering
	\begin{minipage}[t]{0.48\textwidth}
		\centering
		\includegraphics[height=0.28 \textheight]{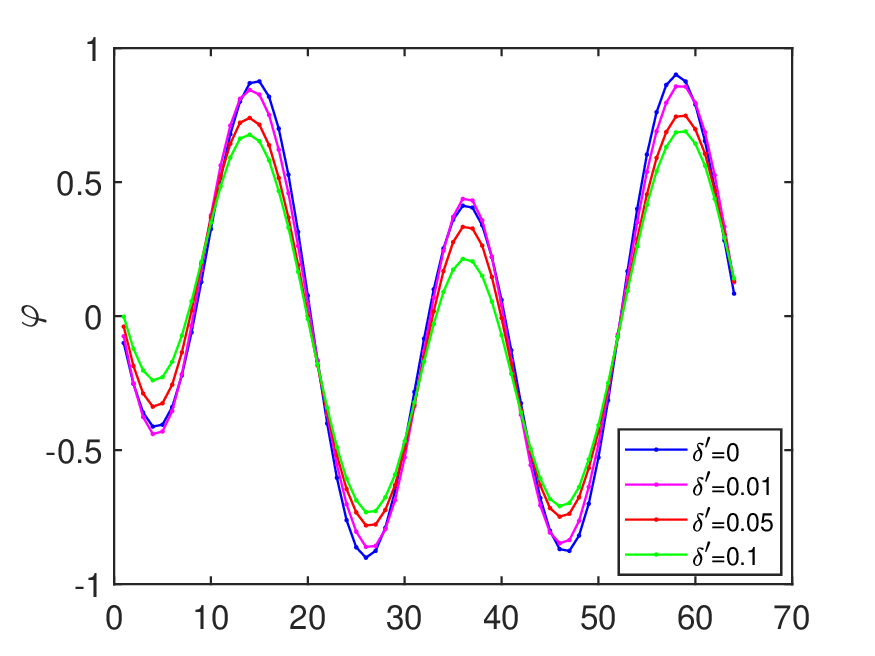}
	\end{minipage}
	\begin{minipage}[t]{0.48\textwidth}
		\centering
		\includegraphics[height=0.28 \textheight]{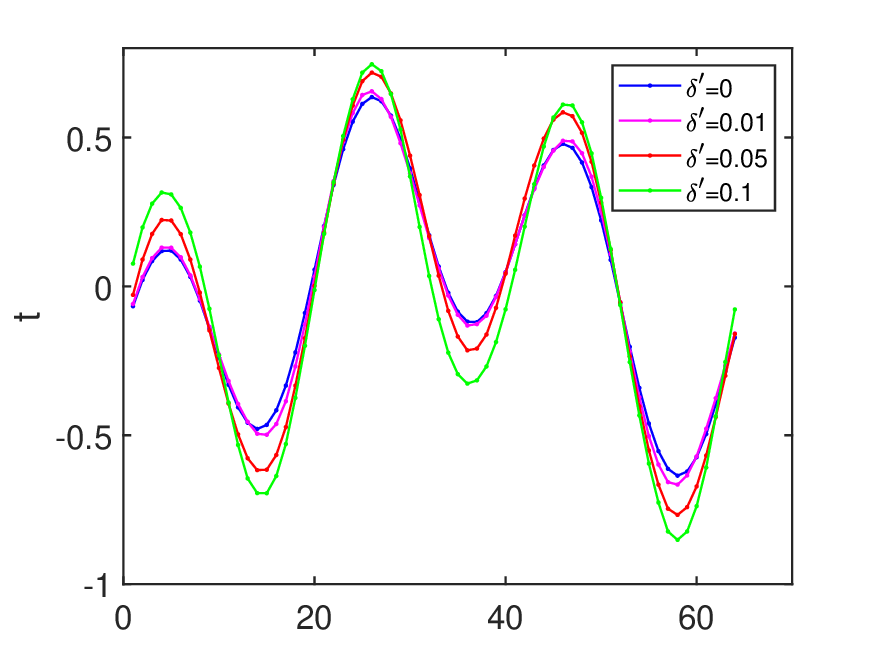}
	\end{minipage}
	\caption{$\varphi^h_{k(\delta)}$ (left) and $t^h_{k(\delta)}$ (right) for different $\delta^\prime$\,(Example 4).}\label{Fig:41}
\end{figure}

\section{Conclusions}\label{sec:con}

This paper proposes a new CCBM-based iterative algorithm for the infinite dimensional Cauchy problem. The iteration is produced through the so-called GKB process. Unlike many other well-known iterative methods such as the Landweber method, $\nu-$method and Nesterov method, where some parameters are introduced and they need to be chosen carefully, no additional parameter needs to be introduced. The sole parameter is the iterative step $k$, which plays the role of the regularization parameter, and is chosen according to the discrepancy principle due to the ill-posedness of the underlying problem. The major innovation lies in the combination of the domain-fitting CCBM framework and the continuous version of the GKB process. As shown by the theoretical analysis and in numerical results, the proposed algorithm is simple but works pretty well, in both accuracy and efficiency. In our opinion, the developed approach can also be easily extended to other inverse problems such as the inverse Robin problem and time-dependent inverse problems for partial differential equations.

\section{Acknowledgement}\label{sec:ack}

This work was supported by the National Natural Science Foundation of China (No. 12071215, 11971230 \& 12171036), the Beijing Natural Science Foundation, China (Key project No. Z210001),  and the National Key Research and Development Program of China (No. 2022YFC3310300).

\bibliographystyle{model1-num-names}
\bibliography{Refs}

\begin{thebibliography}{37}
\expandafter\ifx\csname natexlab\endcsname\relax\def\natexlab#1{#1}\fi
\providecommand{\url}[1]{\texttt{#1}}
\providecommand{\href}[2]{#2}
\providecommand{\path}[1]{#1}
\providecommand{\DOIprefix}{doi:}
\providecommand{\ArXivprefix}{arXiv:}
\providecommand{\URLprefix}{URL: }
\providecommand{\Pubmedprefix}{pmid:}
\providecommand{\doi}[1]{\href{http://dx.doi.org/#1}{\path{#1}}}
\providecommand{\Pubmed}[1]{\href{pmid:#1}{\path{#1}}}
\providecommand{\bibinfo}[2]{#2}
\ifx\xfnm\relax \def\xfnm[#1]{\unskip,\space#1}\fi
\bibitem[{Andrieux et~al.(2006)Andrieux, Baranger, and {Ben Abda}}]{AN}
\bibinfo{author}{S.~Andrieux}, \bibinfo{author}{T.~N. Baranger},
  \bibinfo{author}{A.~{Ben Abda}},
\newblock \bibinfo{title}{Solving {C}auchy problems by minimizing an
  energy{\textendash}like functional},
\newblock \bibinfo{journal}{Inverse Probl.} \bibinfo{volume}{22}
  (\bibinfo{year}{2006}) \bibinfo{pages}{115--133}.
\bibitem[{Andrieux et~al.(2005)Andrieux, {Ben Abda}, and Baranger}]{ABN}
\bibinfo{author}{S.~Andrieux}, \bibinfo{author}{A.~{Ben Abda}},
  \bibinfo{author}{T.~N. Baranger},
\newblock \bibinfo{title}{Data completion via an energy error functional},
\newblock \bibinfo{journal}{C.R. M\'{e}canique} \bibinfo{volume}{333}
  (\bibinfo{year}{2005}) \bibinfo{pages}{171--177}.
\bibitem[{Blum(1989)}]{JIR}
\bibinfo{author}{J.~Blum}, \bibinfo{title}{Numerical {S}imulation and {O}ptimal
  {C}ontrol in {P}lasma {P}hysics with {A}pplications to {T}okamaks},
  \bibinfo{publisher}{Wiley}, \bibinfo{address}{Chichester},
  \bibinfo{year}{1989}.
\bibitem[{Bourgeois(1998)}]{Bour}
\bibinfo{author}{L.~Bourgeois}, \bibinfo{title}{Contr\^{o}le optimal et
  probl\`{e}mes inverses en plasticit\'{e} [Optimal control and Inverse
  Problems in Plasticity]}, \bibinfo{publisher}{Th\`{e}se de Doctorat},
  \bibinfo{address}{Paris}, \bibinfo{year}{1998}.
\bibitem[{Franzone and Magenes(1979)}]{CM}
\bibinfo{author}{P.~C. Franzone}, \bibinfo{author}{E.~Magenes},
\newblock \bibinfo{title}{On the inverse potential problem of
  electrocardiology},
\newblock \bibinfo{journal}{Calcolo} \bibinfo{volume}{16}
  (\bibinfo{year}{1979}) \bibinfo{pages}{459--538}.
\bibitem[{Oliveros{\textendash}Oliveros
  et~al.(2014)Oliveros{\textendash}Oliveros, Morín-Castillo, Aquino-Camacho,
  and Fraguela{\textendash}Collar}]{Analysis2014}
\bibinfo{author}{J.~Oliveros{\textendash}Oliveros},
  \bibinfo{author}{M.~Morín-Castillo}, \bibinfo{author}{F.~Aquino-Camacho},
  \bibinfo{author}{A.~Fraguela{\textendash}Collar},
\newblock \bibinfo{title}{Analysis of the inverse electroencephalographic
  problem for volumetric dipolar sources using a simplification},
\newblock \bibinfo{journal}{Rev. Mex. Ing. Biomed.} \bibinfo{volume}{35}
  (\bibinfo{year}{2014}) \bibinfo{pages}{115--124}.
\bibitem[{Inglese(1997)}]{Ingl}
\bibinfo{author}{G.~Inglese},
\newblock \bibinfo{title}{An inverse problem in corrosion detection},
\newblock \bibinfo{journal}{Inverse Probl.} \bibinfo{volume}{13}
  (\bibinfo{year}{1997}) \bibinfo{pages}{977--994}.
\bibitem[{{Ben Belgacem} and {El Fekih}(2005)}]{BE}
\bibinfo{author}{F.~{Ben Belgacem}}, \bibinfo{author}{H.~{El Fekih}},
\newblock \bibinfo{title}{On {C}auchy{\textquotesingle}s problem: I. a
  variational {S}teklov{\textendash}{P}oincar{\'{e}} theory},
\newblock \bibinfo{journal}{Inverse Probl.} \bibinfo{volume}{21}
  (\bibinfo{year}{2005}) \bibinfo{pages}{1915--1936}.
\bibitem[{{Ben Belgacem}(2007)}]{Belg}
\bibinfo{author}{F.~{Ben Belgacem}},
\newblock \bibinfo{title}{Why is the {C}auchy problem severely
  ill{\textendash}posed?},
\newblock \bibinfo{journal}{Inverse Probl.} \bibinfo{volume}{23}
  (\bibinfo{year}{2007}) \bibinfo{pages}{823--836}.
\bibitem[{Alessandrini et~al.(2009)Alessandrini, Rondi, Rosset, and
  Vessella}]{ARRV}
\bibinfo{author}{G.~Alessandrini}, \bibinfo{author}{L.~Rondi},
  \bibinfo{author}{E.~Rosset}, \bibinfo{author}{S.~Vessella},
\newblock \bibinfo{title}{The stability for the {C}auchy problem for elliptic
  equations},
\newblock \bibinfo{journal}{Inverse Probl.} \bibinfo{volume}{25}
  (\bibinfo{year}{2009}) \bibinfo{pages}{123004}.
\bibitem[{Caubet and Dard{\'e}(2020)}]{Caubet2020}
\bibinfo{author}{F.~Caubet}, \bibinfo{author}{J.~Dard{\'e}},
\newblock \bibinfo{title}{A dual approach to {K}ohn{\textendash}{V}ogelius
  regularization applied to data completion problem},
\newblock \bibinfo{journal}{Inverse Probl.} \bibinfo{volume}{36}
  (\bibinfo{year}{2020}) \bibinfo{pages}{065008}.
\bibitem[{Cheng et~al.(2014)Cheng, Gong, Han, and Zheng}]{CGH14}
\bibinfo{author}{X.~Cheng}, \bibinfo{author}{R.~Gong},
  \bibinfo{author}{W.~Han}, \bibinfo{author}{X.~Zheng},
\newblock \bibinfo{title}{A novel coupled complex boundary method for solving
  inverse source problems},
\newblock \bibinfo{journal}{Inverse Probl.} \bibinfo{volume}{30}
  (\bibinfo{year}{2014}) \bibinfo{pages}{055002}.
\bibitem[{Cheng et~al.(2016)Cheng, Gong, and Han}]{CGH16}
\bibinfo{author}{X.~Cheng}, \bibinfo{author}{R.~Gong},
  \bibinfo{author}{W.~Han},
\newblock \bibinfo{title}{A coupled complex boundary method for the {C}auchy
  problem},
\newblock \bibinfo{journal}{Inverse Probl Sci Eng.} \bibinfo{volume}{24}
  (\bibinfo{year}{2016}) \bibinfo{pages}{1510--1527}.
\bibitem[{Azaïez et~al.(2006)Azaïez, {Ben Belgacem}, and {El Fekih}}]{ABE}
\bibinfo{author}{M.~Azaïez}, \bibinfo{author}{F.~{Ben Belgacem}},
  \bibinfo{author}{H.~{El Fekih}},
\newblock \bibinfo{title}{On {C}auchy{\textquotesingle}s problem: {II}.
  {C}ompletion, regularization and approximation},
\newblock \bibinfo{journal}{Inverse Probl.} \bibinfo{volume}{22}
  (\bibinfo{year}{2006}) \bibinfo{pages}{1307--1336}.
\bibitem[{Kozlov et~al.(1991)Kozlov, Maz'ya, and Fomin}]{KMF}
\bibinfo{author}{V.~A. Kozlov}, \bibinfo{author}{V.~G. Maz'ya},
  \bibinfo{author}{A.~V. Fomin},
\newblock \bibinfo{title}{An iterative method for solving the {C}auchy problem
  for elliptic equations},
\newblock \bibinfo{journal}{Comput. Math. Phys.} \bibinfo{volume}{31}
  (\bibinfo{year}{1991}) \bibinfo{pages}{45--–52}.
\bibitem[{Leitão(2000)}]{Leit}
\bibinfo{author}{A.~Leitão},
\newblock \bibinfo{title}{An iterative method for solving elliptic {C}auchy
  problems},
\newblock \bibinfo{journal}{Numer. Funct. Anal. Opt.} \bibinfo{volume}{21}
  (\bibinfo{year}{2000}) \bibinfo{pages}{715--742}.
\bibitem[{Marin(2020)}]{Marin2020}
\bibinfo{author}{L.~Marin},
\newblock \bibinfo{title}{Landweber{\textendash}{F}ridman algorithms for the
  {C}auchy problem in steady{\textendash}state anisotropic heat conduction},
\newblock \bibinfo{journal}{Math. Mech. Solids} \bibinfo{volume}{25}
  (\bibinfo{year}{2020}) \bibinfo{pages}{1340--1363}.
\bibitem[{Voinea{\textendash}Marinescu
  et~al.(2021)Voinea{\textendash}Marinescu, Marin, and Delvare}]{VLD2021}
\bibinfo{author}{A.-P. Voinea{\textendash}Marinescu},
  \bibinfo{author}{L.~Marin}, \bibinfo{author}{F.~Delvare},
\newblock \bibinfo{title}{{BEM}{\textendash}fading regularization algorithm for
  {C}auchy problems in 2{D} anisotropic heat conduction},
\newblock \bibinfo{journal}{Numer Algorithms.} \bibinfo{volume}{88}
  (\bibinfo{year}{2021}) \bibinfo{pages}{1667--1702}.
\bibitem[{Bucataru et~al.(2022)Bucataru, C\^{ı}mpean, and Marin}]{BCM2022}
\bibinfo{author}{M.~Bucataru}, \bibinfo{author}{I.~C\^{ı}mpean},
  \bibinfo{author}{L.~Marin},
\newblock \bibinfo{title}{A gradient{\textendash}based regularization algorithm
  for the {C}auchy problem in steady{\textendash}state anisotropic heat
  conduction},
\newblock \bibinfo{journal}{Comput. Math. Appl.} \bibinfo{volume}{119}
  (\bibinfo{year}{2022}) \bibinfo{pages}{220--240}.
\bibitem[{Qian et~al.(2008)Qian, Fu, and Li}]{QFL}
\bibinfo{author}{Z.~Qian}, \bibinfo{author}{C.~Fu}, \bibinfo{author}{Z.~Li},
\newblock \bibinfo{title}{Two regularization methods for a {C}auchy problem for
  the {L}aplace equation},
\newblock \bibinfo{journal}{J. Math. Anal. Appl.} \bibinfo{volume}{338}
  (\bibinfo{year}{2008}) \bibinfo{pages}{479--489}.
\bibitem[{Paige and Sanders(1982)}]{PS}
\bibinfo{author}{C.~Paige}, \bibinfo{author}{M.~A. Sanders},
\newblock \bibinfo{title}{{LSQR}: an algorithm for sparse linear equation and
  sparse least squares},
\newblock \bibinfo{journal}{ACM Trans. Math. Softw.} \bibinfo{volume}{8}
  (\bibinfo{year}{1982}) \bibinfo{pages}{43--71}.
\bibitem[{Dautray and Lions(1988)}]{DL}
\bibinfo{author}{R.~Dautray}, \bibinfo{author}{J.~L. Lions},
  \bibinfo{title}{Mathematical {A}nalysis and {N}umerical {M}ethods for
  {S}cience and {T}echnology {V}ol. 2}, \bibinfo{publisher}{Springer},
  \bibinfo{address}{Berlin}, \bibinfo{year}{1988}.
\bibitem[{Isakov(1998)}]{Isakov98}
\bibinfo{author}{V.~Isakov}, \bibinfo{title}{Inverse {P}roblems for {P}artial
  {D}ifferential {E}quations}, \bibinfo{publisher}{Springer},
  \bibinfo{address}{New {Y}ork}, \bibinfo{year}{1998}.
\bibitem[{Han et~al.(2006)Han, Cong, and Wang}]{HCW}
\bibinfo{author}{W.~Han}, \bibinfo{author}{W.~Cong}, \bibinfo{author}{G.~Wang},
\newblock \bibinfo{title}{Mathematical theory and numerical analysis of
  bioluminescence tomography},
\newblock \bibinfo{journal}{Inverse Probl.} \bibinfo{volume}{22}
  (\bibinfo{year}{2006}) \bibinfo{pages}{1659--1675}.
\bibitem[{Karimi and Jozi(2015)}]{KJ}
\bibinfo{author}{S.~Karimi}, \bibinfo{author}{M.~Jozi},
\newblock \bibinfo{title}{A new iterative method for solving linear {F}redholm
  integral equations using the least squares method},
\newblock \bibinfo{journal}{Appl. Math. Comput.} \bibinfo{volume}{250}
  (\bibinfo{year}{2015}) \bibinfo{pages}{744--758}.
\bibitem[{Golub et~al.(1981)Golub, Luk, and Overton}]{GLO}
\bibinfo{author}{G.~H. Golub}, \bibinfo{author}{F.~T. Luk},
  \bibinfo{author}{M.~L. Overton},
\newblock \bibinfo{title}{A block {L}anczos method for computing the singular
  values and corresponding singular vectors of a matrix},
\newblock \bibinfo{journal}{ACM Trans. Math. Softw.} \bibinfo{volume}{7}
  (\bibinfo{year}{1981}) \bibinfo{pages}{149--169}.
\bibitem[{O{\textquotesingle}Leary and Simmons(1981)}]{OS}
\bibinfo{author}{D.~P. O{\textquotesingle}Leary}, \bibinfo{author}{J.~A.
  Simmons},
\newblock \bibinfo{title}{A bidiagonalization{\textendash}regularization
  procedure for large scale discretizations of ill{\textendash}posed problems},
\newblock \bibinfo{journal}{SIAM J. Sci. Stat. Comput.} \bibinfo{volume}{2}
  (\bibinfo{year}{1981}) \bibinfo{pages}{474--489}.
\bibitem[{Caruso and Novati(2019)}]{CN19}
\bibinfo{author}{N.~A. Caruso}, \bibinfo{author}{P.~Novati},
\newblock \bibinfo{title}{Convergence analysis of {LSQR} for compact operator
  equations},
\newblock \bibinfo{journal}{Linear Algebra Appl.} \bibinfo{volume}{583}
  (\bibinfo{year}{2019}) \bibinfo{pages}{146--164}.
\bibitem[{Nemirovskij and Polyak(1984)}]{NP}
\bibinfo{author}{A.~S. Nemirovskij}, \bibinfo{author}{B.~T. Polyak},
\newblock \bibinfo{title}{Iterative methods for solving linear
  ill{\textendash}posed problems under precise information i},
\newblock \bibinfo{journal}{Eng. Cybern.} \bibinfo{volume}{22}
  (\bibinfo{year}{1984}) \bibinfo{pages}{13--25}.
\bibitem[{Jia(2020)}]{Jia}
\bibinfo{author}{Z.~Jia},
\newblock \bibinfo{title}{Approximation accuracy of the {K}rylov subspaces for
  linear discrete ill{\textendash}posed problems},
\newblock \bibinfo{journal}{J. Comput. Appl. Math.} \bibinfo{volume}{374}
  (\bibinfo{year}{2020}) \bibinfo{pages}{112786}.
\bibitem[{Belgacem et~al.(2022)Belgacem, Girault, and Jelassi}]{BGF2022}
\bibinfo{author}{F.~B. Belgacem}, \bibinfo{author}{V.~Girault},
  \bibinfo{author}{F.~Jelassi},
\newblock \bibinfo{title}{Full discretization of {C}auchy's problem by
  {L}avrentiev{\textendash}finite element method},
\newblock \bibinfo{journal}{SIAM J. Numer. Anal.} \bibinfo{volume}{60}
  (\bibinfo{year}{2022}) \bibinfo{pages}{558--584}.
\bibitem[{Mera et~al.(2000)Mera, Elliott, Ingham, and Lesnic}]{MEIL2000}
\bibinfo{author}{N.~S. Mera}, \bibinfo{author}{L.~Elliott},
  \bibinfo{author}{D.~B. Ingham}, \bibinfo{author}{D.~Lesnic},
\newblock \bibinfo{title}{The boundary element solution of the {C}auchy steady
  heat conduction problems in an anisotropic medium},
\newblock \bibinfo{journal}{Int. J. Numer. Methods Eng.} \bibinfo{volume}{49}
  (\bibinfo{year}{2000}) \bibinfo{pages}{481--499}.
\bibitem[{Altiero and Gavazza(1980)}]{AG1980}
\bibinfo{author}{N.~J. Altiero}, \bibinfo{author}{S.~D. Gavazza},
\newblock \bibinfo{title}{On a unified boundary{\textendash}integral equation
  method},
\newblock \bibinfo{journal}{J. Elast.} \bibinfo{volume}{10}
  (\bibinfo{year}{1980}) \bibinfo{pages}{1--9}.
\bibitem[{S(2017)}]{S2017}
\bibinfo{author}{Y.~S},
\newblock \bibinfo{title}{Indirect boundary integral equation method for the
  {C}auchy problem of the {L}aplace equation},
\newblock \bibinfo{journal}{J. Sci. Comput.} \bibinfo{volume}{71}
  (\bibinfo{year}{2017}) \bibinfo{pages}{469--498}.
\bibitem[{Chakib and Nachaoui(2006)}]{CN06}
\bibinfo{author}{A.~Chakib}, \bibinfo{author}{A.~Nachaoui},
\newblock \bibinfo{title}{Convergence analysis for finite element approximation
  to an inverse {C}auchy problem},
\newblock \bibinfo{journal}{Inverse Probl.} \bibinfo{volume}{22}
  (\bibinfo{year}{2006}) \bibinfo{pages}{1191--1206}.
\bibitem[{Kabanikhin and Karchevsky(1995)}]{KK}
\bibinfo{author}{S.~I. Kabanikhin}, \bibinfo{author}{A.~L. Karchevsky},
\newblock \bibinfo{title}{Optimizational method for solving the {C}auchy
  problem for an elliptic equation},
\newblock \bibinfo{journal}{J. Inverse Ill-Posed Probl.} \bibinfo{volume}{3}
  (\bibinfo{year}{1995}) \bibinfo{pages}{21--46}.
\bibitem[{Wedin(1973)}]{PAW}
\bibinfo{author}{P.~Wedin},
\newblock \bibinfo{title}{Perturbation theory for pseudo{\textendash}inverses},
\newblock \bibinfo{journal}{BIT} \bibinfo{volume}{13} (\bibinfo{year}{1973})
  \bibinfo{pages}{217--232}.

\end{thebibliography}

\setcounter{equation}{0}
\renewcommand\theequation{A.\arabic{equation}}

\section*{Appendix}
\subsection*{A.1. Proof of Lemma \ref{injective}}

Let $\mathcal{A}=K^*K$. We need to show that if $K^*K\phi=0$, then $\phi=0$. To this end, let $\textbf{u}=u_1+i\,u_2\in \textbf{V}$ be the weak solution of the forward BVP
\begin{eqnarray*}
	\left\{
	\begin{array}{ll}
		- \Delta \textbf{u} = 0 & \textrm{in}\ \Omega,\\
		\partial_n \textbf{u} + i\,\textbf{u} = 0 & \textrm{on}\ \Gamma _m, \\
		\partial_n \textbf{u} + i\,\textbf{u} = \phi & \textrm{on}\ \Gamma _u.
	\end{array} \right.\label{forwardbvp1}
\end{eqnarray*}
Then $K\phi = u_2$. Let $\textbf{w}\in \textbf{V}$ be the weak solution of the adjoint BVP
\begin{eqnarray*}
	\left\{
	\begin{array}{ll}
		- \Delta \textbf{w} = u_2 & \textrm{in}\ \Omega,\\
		\partial_n \textbf{w} + i\,\textbf{w} = 0 & \textrm{on}\ \Gamma.
	\end{array} \right.\label{adjointeqn2}
\end{eqnarray*}
Then $K^*K\phi=w_2+i\,w_1\vert_{\Gamma_u}\in \textbf{Q}_{\Gamma_u}$. If $K^*K\phi=0$, then $w_1=w_2=0$ on $\Gamma_u$. By using arguments similar to those in Proposition \ref{prop:K}(ii), we have $w_2=0$ in $\Omega$ and $w_1$ satisfies
\begin{eqnarray*}
	\left\{
	\begin{array}{ll}
		- \Delta w_1 = u_2 & \textrm{in}\ \Omega,\\
		w_1 = \partial_n w_1=0 & \textrm{on}\ \Gamma. \\
	\end{array} \right.\label{isp1}
\end{eqnarray*}
Performing integration by parts,
\[
\int_\Omega u_2^2dx = \int_\Omega (- \Delta w_1) u_2dx=\int_\Gamma (\partial_n u_2 w_1 - \partial_n w_1 u_2)ds=0
\]
which indicates $u_2=0$ in $\Omega$, that is, $K\phi=0$. Due to the injectivity of $K$, we obtain $\phi=0$, and the proof is completed.

\subsection*{A.2. A boundary fitting framework}

There is an alternative method of reformulating the Cauchy problem (\cite{CN06,KK}).
\begin{problem}\label{prob:A1}
	With $(\Phi, T)\in H^{-1/2}(\Gamma_m)\times H^{1/2}(\Gamma_m)$, find $t\in H^{1/2}(\Gamma_u)$ such that, 
	\[
	u=T\ \textmd{on}\ \Gamma_m,
	\]
	where $u=u(\Phi,t)$ solves: 
	\begin{equation}
		\left\{
		\begin{array}{ll}
			- \Delta u = 0 & \textrm{in}\ \Omega,\\
			\partial_nu = \Phi & \textrm{on}\ \Gamma_m, \\
			u= t & \textrm{on}\ \Gamma_u.
		\end{array} \right.\label{A11}
	\end{equation}
\end{problem}

In the framework of Problem \ref{prob:A1}, the Dirichlet data is taken as the control variable. Once $t$ is recovered, the Neumann data on the inaccessible boundary is computed numerically through
\[
\varphi = \partial_n u\vert_{\Gamma_u}.
\]

For any $t\in H^{1/2}({\Gamma_u})$, $\Phi \in H^{-1/2}({\Gamma_m})$, we denote by $\tilde{u}=u(0,t)\in V$ and $\hat{u}=u(\Phi,0)\in V$ the weak solutions of the problem
\begin{equation}
	\left\{
	\begin{array}{ll}
		- \Delta \tilde{u} = 0 & \textrm{in}\ \Omega,\\
		\partial_n \tilde{u} = 0 & \textrm{on}\ \Gamma_m, \\
		\tilde{u} = t & \textrm{on}\ \Gamma_u,
	\end{array} \right.\label{A12}
\end{equation}
and the problem
\begin{equation}
	\left\{
	\begin{array}{ll}
		- \Delta \hat{u} = 0 & \textrm{in}\ \Omega,\\
		\partial_n \hat{u} = \Phi & \textrm{on}\ \Gamma_m, \\
		\hat{u} = 0 & \textrm{on}\ \Gamma_u.
	\end{array} \right.\label{A13}
\end{equation}
Then for any $t\in H^{1/2}({\Gamma_u})$, $\Phi \in H^{-1/2}({\Gamma_m})$, we have $u=\hat{u}+\tilde{u}$, and
\[
u=T \ \textmd{on}\ \Gamma_m
\]
reduces to
\[
\tilde{u} =f:= T- \hat{u} \ \textmd{on}\ \Gamma_m.
\]
Considering the noise, $f$ is modified to $f^\delta:=T^\delta- \hat{u}^\delta$, where $\hat{u}^\delta\in V$ is the weak solution of the problem (\ref{A13}) with $\Phi$ being replaced by noisy one $\Phi^\delta$. The noisy data is assumed to belong to the natural space $Q_{\Gamma_m}$ and satisfying
\[
\|\Phi^\delta - \Phi\|_{0,\Gamma_m} \leq \delta, \quad
\|T^\delta - T\|_{0,\Gamma_m} \leq \delta.
\]
Correspondingly, for any $t\in H^{1/2}({\Gamma_u})$, we view the trace $\tilde{u}\vert_{\Gamma_m}\in H^{1/2}({\Gamma_m})$ of $\tilde{u}\in V$, the weak solution of the problem (\ref{A12}), as an element in $Q_{\Gamma_m}$, and define a linear operator $K$ from $H^{1/2}({\Gamma_u})$ to $Q_{\Gamma_m}$ through
\[
t \to K\,t := \tilde{u}|_{\Gamma_m}.
\]
As a result, Problem \ref{prob:A1} can be transformed further, to the following operator equation:
\begin{equation}
	K \phi = f^\delta.\label{A4}
\end{equation}
The adjoint operator $K^*: Q_{\Gamma_u}\to H^{-1/2}({\Gamma_u})$ of $K$ is defined by $K^*v=-\partial_n w\vert_{\Gamma_u}$ for any $v\in Q_{\Gamma_m}$, where $w\in V$ is the weak solution of the adjoint problem
\[
\left\{
\begin{array}{ll}
	- \Delta w = 0 & \textrm{in}\ \Omega,\\
	\partial_n w = v & \textrm{on}\ \Gamma_m, \\
	w = 0 & \textrm{on}\ \Gamma_u.
\end{array} \right.
\]

\subsection*{A.3. A domain fitting framework}

Another domain fitting framework for the Cauchy problem is based on the reformulation of it as follows (\cite{AN}):
\begin{problem}\label{prob:B}
	With $(\Phi, T)\in H^{-1/2}(\Gamma_m)\times H^{1/2}(\Gamma_m)$, find $(\varphi,t)\in H^{-1/2}(\Gamma_u)\times H^{1/2}(\Gamma_u)$ such that
	\[
	u_1=u_2\ \textmd{in}\ \Omega,
	\]
	where $u_1=u_1(T,\varphi)$ and $u_2=u_2(\Phi,t)$ solve the BVPs
	\begin{equation}
		\left\{
		\begin{array}{ll}
			- \Delta u_1 = 0 & \textrm{in}\ \Omega,\\
			u_1 = T & \textrm{on}\ \Gamma_m, \\
			\partial_n u_1= \varphi & \textrm{on}\ \Gamma_u,
		\end{array} \right.\label{B1}
	\end{equation}
	and
	\begin{equation}
		\left\{
		\begin{array}{ll}
			- \Delta u_2 = 0 & \textrm{in}\ \Omega,\\
			\partial_n u_2= \Phi & \textrm{on}\ \Gamma_m, \\
			u_2 = t & \textrm{on}\ \Gamma_u,
		\end{array} \right.\label{B2}
	\end{equation}
	respectively.
\end{problem}

For any $\varphi \in H^{-1/2}({\Gamma_u})$, $T\in H^{1/2}({\Gamma_m})$, we denote by $\tilde{u}_1=u_1(0,\varphi), \hat{u}_1=u_1(T,0)\in V$. Then we have $u_1=\hat{u}_1+\tilde{u}_1$. Similarly, for any $t \in H^{1/2}({\Gamma_u})$, $\Phi \in H^{-1/2}({\Gamma_m})$, it holds that $u_2=\hat{u}_2+\tilde{u}_2$ with $\tilde{u}_2=u_2(0,t), \hat{u}_2=u_2(\Phi,0)\in V$.
Therefore,
\[
u_1=u_2\ \textmd{in}\ \Omega
\]
reduces to
\[
\tilde{u}_1- \tilde{u}_2= f:= \hat{u}_2- \hat{u}_1\ \textmd{in}\ \Omega.
\]
With noisy data $(\Phi^\delta, T^\delta)$, $f$ is modified to $f^\delta:=\hat{u}^\delta_2- \hat{u}^\delta_1$, where $\hat{u}^\delta_1=u_1(T^\delta,0), \hat{u}^\delta_2=u_2(\Phi^\delta,0)\in V$, and
\[
\|\Phi^\delta - \Phi\|_{-1/2,\Gamma_m} \leq \delta, \quad
\|T^\delta - T\|_{1/2,\Gamma_m} \leq \delta.
\]

Define a linear operator $K$ from $H^{-1/2}({\Gamma_u})\times H^{1/2}({\Gamma_u})$ to $Q$ through
\[
\phi \to K\,\phi := \tilde{u}_1-\tilde{u}_2,
\]
where for any $\phi=(\varphi,t)\in H^{-1/2}({\Gamma_u})\times H^{1/2}({\Gamma_u})$, $\tilde{u}_1=u_1(0,\varphi), \tilde{u}_2=u_1(0,t)\in V$ and is viewed as an element in $Q$.
As a result, Problem \ref{prob:B} can be transformed further, to the following operator equation:
\begin{equation}
	K \phi = f^\delta.\label{B3}
\end{equation}
The adjoint operator $K^*: Q\to H^{1/2}({\Gamma_u})\times H^{-1/2}({\Gamma_u})$ of $K$ is defined by $K^*v=(w_1, \partial_n w_2)$ for any $v\in Q$, where $w_1, w_2\in V$ are the weak solutions of the adjoint problems
\[
\left\{
\begin{array}{ll}
	- \Delta w_1 = v & \textrm{in}\ \Omega,\\
	w_1 = 0 & \textrm{on}\ \Gamma_m, \\
	\partial_n w_1 = 0 & \textrm{on}\ \Gamma_u,
\end{array} \right.
\]
and
\[
\left\{
\begin{array}{ll}
	- \Delta w_2 = v & \textrm{in}\ \Omega,\\
	\partial_n w_2 = 0 & \textrm{on}\ \Gamma_m, \\
	w_2 = 0 & \textrm{on}\ \Gamma_u,
\end{array} \right.
\]
respectively.

\subsection*{A.4. Proof of Lemma \ref{prop:Gk}}

We prove the Lemma by contradiction. By the second equation of (\ref{gkbmatix}), we have $\beta_j = (Kp_j,v_j)_{0,\Omega}$. Assume that for some positive constant $\delta$, the set $H = \{k\in\mathbb{N}: (Kp_k,v_k)_{0,\Omega} \geq 2\delta\}$ is infinite. Therefore $H$ has a countable subset which, by a change of notation, we can identify with $\mathbb{N}$. Thus $\{p_n\}^\infty_{n=1}$ is an orthonormal sequence and
\begin{equation}
	(Kp_n,v_n)_{0,\Omega} \geq 2\delta,\ n \geq 1.\label{beta}
\end{equation}
Since $K$ is compact, there is a subsequence $\{n(j)\}_{j \geq 1}$ such that $\{Kp_{n(j)}\} \to z \in Q$. By deleting a finite number of terms from this sequence, we may suppose that $\|Kp_{n(j)}-z \|_{0,\Omega} < \delta$ for all $j \geq 1$. Thus, $$\vert (Kp_{n(j)},v_{n(j)})_{0,\Omega} - (z,v_{n(j)})_{0,\Omega} \vert = \vert (Kp_{n(j)}-z,v_{n(j)})_{0,\Omega} \vert\leq \| Kp_{n(j)}-z \|_{0,\Omega} < \delta.$$
By (\ref{beta}) and using the reverse triangle inequality, $\vert \vert a\vert-\vert b\vert \vert \leq \vert a-b\vert$, we obtain
\[
\vert (z,v_{n(j)})_{0,\Omega} \vert > \delta,
\]
so that the series $\sum_{j \geq 1}\vert (z,v_{n(j)})_{0,\Omega} \vert^2$ diverges. This is a contradiction because $\{v_{n(j)}\}$ is an orthonormal system and therefore $\sum_{j \geq 1}\vert (z,v_{n(j)})_{0,\Omega} \vert^2 \leq \|z\|_{0,\Omega}$.

By the second equation of (\ref{gkbmatix}), we also have $\gamma_{j+1} = (Kp_{j},v_{j+1})_{0,\Omega}$. the same arguments ensure that $\{\gamma_{j+1}\}_{j\geq 2}\to 0$.

\subsection*{A.5. Proof of Lemma \ref{lemmaParameter}}

We prove the Lemma by mathematical induction. For $n = 1$, it follows from (\ref{right-handsidebound}) and the formulae for $\gamma_1^\delta, \gamma_1, v_1^\delta, v_1, \beta_1^\delta, \beta_1, p_1^\delta$ and $p_1$ in (\ref{gkb}) that we can yield
\begin{align*}
	&\vert\gamma_1^\delta - \gamma_1\vert = \vert \|f^\delta\|_{0,\Omega}  - \|f\|_{0,\Omega}  \vert \leq \|f^\delta - f\|_{0,\Omega} \leq C_f\delta,\\
	&\|v_1^\delta - v_1\|_{0,\Omega} = \| \frac{f^\delta}{\gamma_1^\delta} - \frac{f}{\gamma_1} \|_{0,\Omega} = \| \frac{(\gamma_1-\gamma_1^\delta)f^\delta+\gamma_1^\delta(f^\delta-f)}{\gamma_1^\delta\gamma_1} \|_{0,\Omega}\\
	& \leq \frac{1}{\gamma_1^\delta\gamma_1}(\vert\gamma_1^\delta - \gamma_1\vert\|f^\delta\|_{0,\Omega}+\|f^\delta - f\|_{0,\Omega}\gamma_1^\delta)\\
	& \leq \frac{1}{\gamma_1^\delta\gamma_1}(\|f^\delta - f\|_{0,\Omega}\gamma_1^\delta+\|f^\delta - f\|_{0,\Omega}\gamma_1^\delta) \leq \frac{2}{\gamma_1}C_f\delta \leq 2C_3 C_f\delta,\\
	&\vert \beta_1^\delta - \beta_1 \vert
	= \vert \|K^*v_1^\delta\|_{\textbf{Q}_{\Gamma_u}} - \|K^*v_1\|_{\textbf{Q}_{\Gamma_u}}\vert \leq \|K^*v_1^\delta - K^*v_1\|_{\textbf{Q}_{\Gamma_u}}\\ \hspace{-2.5cm}
	& \leq \|K^*\|\|v_1^\delta - v_1\|_{0,\Omega} \leq 2\|K\|C_3 C_f\delta,\\
	&\|p_1^\delta - p_1\|_{\textbf{Q}_{\Gamma_u}}
	= \|\frac{K^*v_1^\delta}{\beta_1^\delta} - \frac{K^*v_1}{\beta_1}\|_{\textbf{Q}_{\Gamma_u}}
	= \|\frac{(\beta_1 - \beta_1^\delta)K^*v_1^\delta + \beta_1^\delta(K^*v_1^\delta - K^*v_1)}{\beta_1^\delta\beta_1}\|_{\textbf{Q}_{\Gamma_u}} \\ 
	& \leq \frac{1}{\beta_1^\delta\beta_1}(\vert \beta_1 - \beta_1^\delta\vert\|K^*v_1^\delta\|_{\textbf{Q}_{\Gamma_u}} + \beta_1^\delta\|K^*v_1^\delta - K^*v_1\|_{\textbf{Q}_{\Gamma_u}})\\ 
	& \leq \frac{1}{\beta_1^\delta\beta_1}(\|K^*v_1^\delta - K^*v_1\|_{\textbf{Q}_{\Gamma_u}}\beta_1^\delta + \beta_1^\delta\|K^*v_1^\delta - K^*v_1\|_{\textbf{Q}_{\Gamma_u}})
\end{align*}
\begin{align*}
	\hspace{-8em}
	& \leq \frac{2\|K^*\|\|v_1^\delta - v_1\|_{0,\Omega}}{\beta_1}
	\leq 4\|K\|C_3^2C_f\delta.
\end{align*}
Thus, we obtain
\begin{align*}
	&\vert\gamma_1^\delta - \gamma_1\vert + \|v_1^\delta - v_1\|_{0,\Omega} + \vert \beta_1^\delta - \beta_1 \vert + \|p_1^\delta - p_1\|_{\textbf{Q}_{\Gamma_u}}\\
	&\leq (1+2C_3+2\|K\|C_3 +4\|K\|C^2_3)C_f\delta,
\end{align*}
which yields (\ref{lemmaParameterIneq}) for $k=1$ if $C_4\geq 1+2C_3+2\|K\|C_3 +4\|K\|C^2_3$.

In a similar way, a fixed number $C_4$ exists such that the inequality (\ref{lemmaParameterIneq}) holds for $k=2, \cdots, k_0$ with $k_0 = \lfloor (12C^{-2}_3(\|K\| + \beta_1 + 1))^{1/(1-2\theta)} \rfloor$.

Now, assume that for $n = k - 1$ ($k>k_0+1$), the following inequality holds
\begin{align*}
	&\vert \gamma_{k-1}^\delta - \gamma_{k-1} \vert +  \|v_{k-1}^\delta - v_{k-1}\|_{0,\Omega} + \vert \beta_{k-1}^\delta - \beta_{k-1} \vert + \|p_{k-1}^\delta - p_{k-1}\|_{\textbf{Q}_{\Gamma_u}}\\
	&\leq C_4 (k-1)! \delta.
\end{align*}

Next, we will prove that it holds for the case of $n=k$. In this case, by the formulae in (\ref{gkb}), Lemma \ref{prop:Gk}, $\|v_{k-1}^\delta\|_{0,\Omega}=\|p_{k-1}^\delta\|_{\textbf{Q}_{\Gamma_u}}=1$ and using the reverse triangle inequality, $\vert\|b\|-\|a\|\vert\leq \|b-a\|$, we have
\begin{align*}
	&\vert \gamma_{k}^\delta - \gamma_{k} \vert = \vert \|Kp_{k-1}^\delta - \beta_{k-1}^\delta v_{k-1}^\delta\|_{0,\Omega} - \|Kp_{k-1} - \beta_{k-1} v_{k-1}\|_{0,\Omega} \vert\\
	&\leq \|K(p_{k-1}^\delta - p_{k-1})+\beta_{k-1}v_{k-1}-\beta_{k-1}^\delta v_{k-1}^\delta\|_{0,\Omega}\\
	&= \|K(p_{k-1}^\delta - p_{k-1})+\beta_{k-1}(v_{k-1}-v_{k-1}^\delta)+(\beta_{k-1}-\beta_{k-1}^\delta) v_{k-1}^\delta\|_{0,\Omega}\\
	&\leq \|K\|\|p_{k-1}^\delta - p_{k-1}\|_{\textbf{Q}_{\Gamma_u}} + \beta_{k-1} \|v_{k-1}^\delta - v_{k-1}\|_{0,\Omega} + \vert \beta_{k-1}^\delta - \beta_{k-1} \vert \\
	&\leq C_4(\|K\| + \beta_1 + 1)(k-1)!\delta,\\
	&\| v_{k}^\delta - v_{k} \|_{0,\Omega} = \|\frac{Kp_{k-1}^\delta-\beta_{k-1}^\delta v_{k-1}^\delta}{\gamma_k^\delta}-\frac{Kp_{k-1}-\beta_{k-1} v_{k-1}}{\gamma_k}\|_{0,\Omega}\\
	&= \frac{1}{\gamma_k^\delta\gamma_k}\|\gamma_k(Kp_{k-1}^\delta-\beta_{k-1}^\delta v_{k-1}^\delta)-\gamma_k^\delta(Kp_{k-1}-\beta_{k-1}v_{k-1})\|_{0,\Omega}\\
	&=\frac{1}{\gamma_k^\delta\gamma_k}\|(\gamma_k-\gamma_k^\delta)(Kp_{k-1}^\delta-\beta_{k-1}^\delta v_{k-1}^\delta)+\gamma_k^\delta K(p_{k-1}^\delta-p_{k-1})\\
	&\quad+\gamma_k^\delta\beta_{k-1}(v_{k-1}-v_{k-1}^\delta)+\gamma_k^\delta(\beta_{k-1}-\beta_{k-1}^\delta)v_{k-1}^\delta\|_{0,\Omega}\\
	&\leq \frac{1}{\gamma_k}(\vert \gamma_{k}^\delta - \gamma_{k}\vert+\|K\|\|p_{k-1}^\delta - p_{k-1}\|_{\textbf{Q}_{\Gamma_u}}+\beta_{k-1} \|v_{k-1}^\delta - v_{k-1}\|_{0,\Omega}\\
	&\quad+\vert \beta_{k-1}^\delta - \beta_{k-1} \vert)\\
	&\leq \frac{2}{\gamma_k}(\|K\|\|p_{k-1}^\delta - p_{k-1}\|_{\textbf{Q}_{\Gamma_u}} + \beta_{k-1} \|v_{k-1}^\delta - v_{k-1}\|_{0,\Omega} + \vert \beta_{k-1}^\delta - \beta_{k-1} \vert)
\end{align*}
\begin{align*}
	&\leq \frac{2C_4 k^\theta}{C_3}(\|K\| + \beta_1 + 1)(k-1)!\delta,\\
	&\vert \beta_{k}^\delta - \beta_{k} \vert = \vert \|K^*v_k^\delta-\gamma_k^\delta p_{k-1}^\delta\|_{\textbf{Q}_{\Gamma_u}}-\|K^*v_k-\gamma_k p_{k-1}\|_{\textbf{Q}_{\Gamma_u}}\vert\\
	&\leq \|K^*(v_k^\delta-v_k)-\gamma_k^\delta p_{k-1}^\delta+\gamma_k p_{k-1}\|_{\textbf{Q}_{\Gamma_u}}\\
	&=\|K^*(v_k^\delta-v_k)+(\gamma_k-\gamma_k^\delta) p_{k-1}^\delta+\gamma_k(p_{k-1}-p_{k-1}^\delta)\|_{\textbf{Q}_{\Gamma_u}}\\
	&\leq \|K\|\|v_{k}^\delta - v_{k}\|_{0,\Omega}+ \vert \gamma_{k}^\delta - \gamma_{k} \vert + \gamma_{k} \|p_{k-1}^\delta - p_{k-1}\|_{\textbf{Q}_{\Gamma_u}} \\
	&\leq \big[\frac{2k^\theta}{C_3}(\|K\| + \beta_1 + 1) + (\|K\| + \beta_1 + 1) + \gamma_1\big] C_4(k-1)!\delta,\\
	&\| p_{k}^\delta - p_{k} \|_{\textbf{Q}_{\Gamma_u}} =\|\frac{K^*v_k^\delta-\gamma_k^\delta p_{k-1}^\delta}{\beta_k^\delta}-\frac{K^*v_k-\gamma_k p_{k-1}}{\beta_k}\|_{\textbf{Q}_{\Gamma_u}}\\
	&=\frac{1}{\beta_k^\delta\beta_k}\|\beta_k (K^*v_k^\delta-\gamma_k^\delta p_{k-1}^\delta)-\beta_k^\delta (K^*v_k-\gamma_kp_{k-1})\|_{\textbf{Q}_{\Gamma_u}}\\
	&=\frac{1}{\beta_k^\delta\beta_k}\|(\beta_k-\beta_k^\delta)(K^*v_k^\delta-\gamma_k^\delta p_{k-1}^\delta)+\beta_k^\delta K^*(v_k^\delta-v_k)+\beta_k^\delta(\gamma_k-\gamma_k^\delta)p_{k-1}^\delta\\
	&\quad+\beta_k^\delta\gamma_k(p_{k-1}-p_{k-1}^\delta)\|_{\textbf{Q}_{\Gamma_u}}\\
	&\leq \frac{1}{\beta_k}(\vert \beta_{k}^\delta - \beta_{k} \vert + \|K^*\|\|v_{k}^\delta - v_{k}\|_{0,\Omega} + \vert \gamma_{k}^\delta - \gamma_{k} \vert + \gamma_{k} \|p_{k-1}^\delta - p_{k-1}\|_{\textbf{Q}_{\Gamma_u}})\\
	&\leq \frac{2}{\beta_k}(\|K^*\|\|v_{k}^\delta - v_{k}\|_{0,\Omega} +  \vert \gamma_{k}^\delta - \gamma_{k} \vert + \gamma_{k} \|p_{k-1}^\delta - p_{k-1}\|_{\textbf{Q}_{\Gamma_u}})\\
	&\leq \frac{2}{C_3} k^\theta [\frac{2k^\theta}{C_3}(\|K\| + \beta_1 + 1) + (\|K\| + \beta_1 + 1) + \gamma_1] C_4(k-1)!\delta.
\end{align*}

Thus, by combining all four inequalities above, we deduce
\begin{align*}
	&\vert \gamma_k^\delta - \gamma_k \vert +  \|v_k^\delta - v_k\|_{0,\Omega} + \vert \beta_k^\delta - \beta_k \vert + \|p_k^\delta - p_k\|_{\textbf{Q}_{\Gamma_u}}\\
	&\leq \left[ 2(\|K\| + \beta_1 + 1) \left( \frac{2k^{2\theta}}{C^2_3} +  \frac{3k^{\theta}}{C_3} + 1 \right) \right] C_4(k-1)!\delta \\
	&\leq \left[ 2(\|K\| + \beta_1 + 1) \left( \frac{2}{C^2_3k^{1-2\theta}} +  \frac{3}{C_3k^{1-\theta}} + \frac{1}{k} \right) \right] C_4k!\delta,
\end{align*}
which yields the required inequality (\ref{lemmaParameterIneq}) by noting that $2(\|K\|+\beta_1+1)\left(\frac{2}{C^2_3k^{1-2\theta}}\right. \\ \left.+\frac{3}{C_3k^{1-\theta}}+\frac{1}{k} \right)<1$ for $k\geq k_0$.

\subsection*{A.6. Proof of Lemma \ref{ginverse}}

It follows from (\ref{gkbmatix}) that $G_k^\delta = (V_{k+1}^\delta)^*KP_k^\delta$ and $G_k = V_{k+1}^*KP_k$. On the other hand, we have (see e.g. \cite[Theorem 4.1]{PAW})
\begin{equation}
	\label{IneqGinverse1}
	\|B^\dagger - A^\dagger\|_2\leq \sqrt{2}\|B^\dagger\|_2\|A^\dagger\|_2\|B-A\|_2,
\end{equation}
where $A, B\in \mathbb{R}^{k\times j}$ satisfy $\textrm{rank}(A)=\textrm{rank}(B)$. By Lemma \ref{lemmaParameter} and the fact that $\|V_{k+1}^\star\|_{0,\Omega}=1$, we obtain
\begin{align}
	&\|G_k^\delta - G_k\|_2
	= \|(V_{k+1}^\delta)^*KP_k^\delta - V_{k+1}^*KP_k\|_2\nonumber \\
	& \leq \|((V_{k+1}^\delta)^* - V_{k+1}^*)KP_k^\delta + V_{k+1}^*K(P_k^\delta - P_k)\|_2\nonumber\\
	& \leq \|(V_{k+1}^\delta)^* - V_{k+1}^*\|_{0,\Omega}\|K\| + \|P_k^\delta - P_k\|_{\textbf{Q}_{\Gamma_u}}\|K\|\nonumber\\
	& \leq (\sum_{j = 1}^{k+1} \|v_j^\delta - v_j\|_{0,\Omega} + \sum_{i = 1}^{k} \|p_i^\delta - p_i\|_{\textbf{Q}_{\Gamma_u}})\|K\|
	\leq 2C_4(k+1)!\delta.
	\label{matrixInequality}
\end{align}
Combining inequalities (\ref{IneqGinverse1}) and (\ref{matrixInequality}), we complete the proof.

\subsection*{A.7. Proof of Lemma \ref{kRate}}

Now, we prove the first inequality. It follows from $\vert \bar{\mu}_k\vert\leq C_be^{-pk\ln k}$ that
\begin{equation}\label{eq: 6.1}
	k\ln k\leq \frac{1}{p}\ln C_b + \frac{1}{p}\ln (\vert\bar{\mu}_k\vert^{-1}).
\end{equation}
Using the inequality $\frac{p-1}{p}\ln x>p\ln \ln x-p\ln{\frac{p^2}{p-1}}, p>1$, we have
\[
\frac{1}{p}\ln C_b + \frac{1}{p}\ln (\vert\bar{\mu}_k\vert^{-1})\leq \ln C_6 + \ln(\vert\bar{\mu}_k\vert^{-1})-p\ln\ln(\vert\bar{\mu}_k\vert^{-1}),
\]
where $C_6 = C_b^{1/p}(\frac{p^2}{p-1})^p$. From (\ref{eq: 6.1}) and the above inequality, we obtain
\begin{equation}\label{eq: 6.2}
	k^k\leq C_6\vert\bar{\mu}_k\vert^{-1}\ln^{-p}(\vert\bar{\mu}_k\vert^{-1}).
\end{equation}
By Stirling's formula, we yield
\begin{align}\label{eq: 6.3}
	(k+3)!&\leq e\sqrt{2\pi}(k+3)^{k+\frac{7}{2}}e^{-(k+3)}
	= e\sqrt{2\pi}(\frac{k+3}{k})^{k+\frac{7}{2}}k^{k+\frac{7}{2}}e^{-(k+3)}\nonumber\\
	&= e\sqrt{2\pi}[(1+\frac{3}{k})^{\frac{k}{3}}]^{(3+\frac{21}{2k})}k^k k^{\frac{7}{2}}e^{-(k+3)}.
\end{align}
Since $(1+\frac{3}{k})^{k/3}$ increases monotonically and converges to $e$, and $k^{7/2}<e^{k+3}$ for all $k$, we obtain
\[
(k+3)!\leq e\sqrt{2\pi}e^{3+21/(2k)}k^k\leq C_7\vert\bar{\mu}_k\vert^{-1}\ln^{-p}(\vert\bar{\mu}_k\vert^{-1}),
\]
where $C_7 = \sqrt{2\pi}e^{14.5}$.

Next, we prove the second inequality. From $\vert\bar{\mu}_{k+1}\vert\geq C_ae^{-e^{(k+1)/\sigma}}$ and $(k+1)/\sigma >0$, we have
\[
\ln(\vert\bar{\mu}_{k+1}\vert^{-1})\leq -\ln C_a + e^{(k+1)/\sigma}
\leq \max\{\ln C_a^{-1}, 0\}+ e^{(k+1)/\sigma}
\leq C_8e^{(k+1)/\sigma},
\]
where $C_8=1+\max\{\ln C_a^{-1}, 0\}$. Taking logarithms on both sides of the above inequality and multiplying by $p$, we derive the following
\begin{equation}\label{eq: 6.4}
	p\ln\ln(\vert\bar{\mu}_{k+1}\vert^{-1})\leq p\ln C_8+\frac{p}{\sigma}(k+1).
\end{equation}
On the other hand, due to $\sigma\geq p/(2\nu\ln(\rho^{-1}))$, we have
\begin{equation}\label{eq: 6.5}
	p\ln C_8 + \frac{p}{\sigma}(k+1) \leq \ln C_8^p + 2\nu(k+1)\ln(\rho^{-1}).
\end{equation}
Combining (\ref{eq: 6.4}) and (\ref{eq: 6.5}), we obtain
\begin{eqnarray*}
	\rho^{2\nu(k+1)} \leq C_8^p \ln^{-p} (\vert\bar{\mu}_{k+1}\vert ^{-1}).
\end{eqnarray*}
Then, taking $C_5 = \max\{C_7, C_8^p\}$, we obtain (\ref{kinequality}).

\let\thefootnote\relax\footnotetext{Please cite to this paper as published in: 
	
	 Journal of Computational and Applied Mathematics, 432 (2023), pp. 115282, DOI: https://doi.org/10.1016/j.cam.2023.115282.}
\end{document}